\documentclass[12pt]{amsart}
\usepackage{amssymb}
\usepackage{mathrsfs}
\usepackage{stmaryrd}
\usepackage{eucal}
\usepackage[all]{xy}
\usepackage[margin=1in]{geometry} 
\usepackage{hyperref}
\usepackage[dvipsnames]{xcolor}
\usepackage[shortlabels]{enumitem}
\hypersetup{bookmarksdepth=2}
\hypersetup{colorlinks=true}
\hypersetup{linkcolor=blue}
\hypersetup{citecolor=blue}
\hypersetup{urlcolor=blue}

\newtheorem{thma}{Theorem}

\numberwithin{equation}{section}
\newtheorem{theorem}[equation]{Theorem}

\newtheorem{proposition}[equation]{Proposition}
\newtheorem{lemma}[equation]{Lemma}
\newtheorem{corollary}[equation]{Corollary}

\theoremstyle{definition}
\newtheorem{rmk}[equation]{Remark}
\newenvironment{remark}[1][]{\begin{rmk}[#1] \pushQED{\qed}}{\popQED \end{rmk}}
\newtheorem{eg}[equation]{Example}
\newenvironment{example}[1][]{\begin{eg}[#1] \pushQED{\qed}}{\popQED \end{eg}}
\newtheorem{defnaux}[equation]{Definition}

\newcommand{\bA}{\mathbf{A}}

\newcommand{\bG}{\mathbf{G}}

\newcommand{\fL}{\mathfrak{L}}

\newcommand{\cM}{\mathcal{M}}

\newcommand{\cR}{\mathcal{R}}

\newcommand{\bS}{\mathbf{S}}
\newcommand{\cS}{\mathcal{S}}

\newcommand{\bV}{\mathbf{V}}
\newcommand{\cV}{\mathcal{V}}

\newcommand{\bZ}{\mathbf{Z}}

\newcommand{\be}{\mathbf{e}}

\newcommand{\stacks}[1]{\cite[\href{http://stacks.math.columbia.edu/tag/#1}{Tag~#1}]{stacks}}
\newcommand{\DOI}[1]{\href{http://doi.org/#1}{\color{purple}{\tiny\tt DOI:#1}}}
\newcommand{\arxiv}[1]{\href{http://arxiv.org/abs/#1}{{\tiny\tt arXiv:#1}}}

\newcommand{\defn}[1]{\textit{#1}}
\let\ol\overline
\let\ul\underline
\let\lbb\llbracket
\let\rbb\rrbracket
\DeclareMathOperator{\ud}{ud}
\DeclareMathOperator{\ed}{ed}

\DeclareMathOperator{\im}{im}
\DeclareMathOperator{\rank}{rank}
\DeclareMathOperator{\Sym}{Sym}

\DeclareMathOperator{\Spec}{Spec}
\DeclareMathOperator{\Hom}{Hom}
\DeclareMathOperator{\Frac}{Frac}
\DeclareMathOperator{\uIsom}{\underline{Isom}}
\DeclareMathOperator{\uHom}{\underline{Hom}}
\renewcommand{\int}{\operatorname{int}}
\newcommand{\orb}{\mathrm{orb}}
\newcommand{\umu}{\ul{\smash{\mu}}} 
\newcommand{\unu}{\ul{\smash{\nu}}}
\newcommand{\ulambda}{\ul{\smash{\lambda}}}

\newcommand{\type}{\mathrm{type}} 
\newcommand{\lpp}{(\!(}
\newcommand{\rpp}{)\!)}
\newcommand{\GL}{\mathbf{GL}}
\newcommand{\Gr}{\mathbf{Gr}}
\renewcommand{\phi}{\varphi}
\renewcommand{\emptyset}{\varnothing}
\newcommand{\id}{\mathrm{id}}

\title{Improved unirationality for $\GL$-varieties}

\author{Arthur Bik}
\address{D.E.~Shaw \& Co., New York}
\email{\href{mailto:arthur.bik@deshaw.com}{arthur.bik@deshaw.com}}
\urladdr{\url{http://arthurbik.nl}}

\author{Jan Draisma}
\address{University of Bern, Switzerland}
\email{\href{mailto:jan.draisma@math.unibe.ch}{jan.draisma@math.unibe.ch}}
\urladdr{\url{https://math-unibe.ch/jdraisma/}}

\author{Rob Eggermont}
\address{Eindhoven University of Technology, The Netherlands}
\email{\href{mailto:r.h.eggermont@tue.nl}{r.h.eggermont@tue.nl}}
\urladdr{\url{https://www.tue.nl/en/research/researchers/rob-eggermont/}}

\author{Andrew Snowden}
\address{Department of Mathematics, University of Michigan, Ann Arbor, MI}
\email{\href{mailto:asnowden@umich.edu}{asnowden@umich.edu}}
\urladdr{\url{http://www-personal.umich.edu/~asnowden/}}

\thanks{JD was partially supported by Swiss National
Science Foundation project grant 200021-227864.
RE was partially supported by the NWO Veni grant entitled {\em Stability and structure in infinite-dimensional spaces}, project number 016.Veni.192.113.
AS was supported by NSF grants DMS-1453893 and DMS-2301871.
}
  
\date{June 2, 2026}

\begin{document}

\begin{abstract}
A $\GL$-variety is a typically infinite dimensional variety equipped with a suitable
action of the infinite general linear group $\GL$. In earlier work, we established the
unirationality theorem: an irreducible $\GL$-variety admits a dominant map
from a particularly simple $\GL$-variety, namely,
the product of an irreducible finite-dimensional variety with trivial
$\GL$-action and an infinite-dimensional affine space on which $\GL$
acts linearly. The main result of this paper states that this map can in fact be
constructed to be surjective rather than merely dominant. An immediate
application is that secant varieties to varieties of tensors, which are
typically constructed as image closures of certain $\GL$-equivariant maps,
are in fact also images of (more complicated) $\GL$-equivariant maps. We
derive several consequences of this improved unirationality theorem.
\end{abstract}

\maketitle
\tableofcontents

\section{Introduction} \label{s:intro}

\subsection{Overview}

Fix an algebraically closed field $K$ of characteristic~0. A \defn{$\GL$-variety} is a (typically infinite dimensional) affine variety over $K$ equipped with an action of the infinite general linear group $\GL$, such that the coordinate ring is a $\GL$-finitely generated $\GL$-algebra; see \S\ref{ssec:GLvar} and \S\ref{s:bg} for details. In \cite{polygeom}, we introduced this class of varieties and proved several fundamental results about them, and in \cite{imgclosure}, we proved some deeper results. This paper builds on these previous two papers.

One of the important results from \cite{polygeom} is the \emph{unirationality theorem}: if $X$ is an irreducible $\GL$-variety, then there is a dominant $\GL$-equivariant map $\phi \colon B \times \bA^{\ulambda} \to X$, where $B$ is an irreducible finite-dimensional variety with trivial $\GL$-action and $\bA^{\ulambda}$ is an affine space on which $\GL$ acts linearly. Here $\ulambda$ is a finite tuple of permutations that encodes the action of $\GL$; see \S\ref{ssec:GLvar}. The main goal of the current paper is to establish a much improved version of the unirationality theorem, which says that that map $\phi$ can be chosen to be surjective. To see why this is surprising, we first discuss an application to border strength, but we stress that our results are far more general than this application. 

\subsection{Application to border strength}

Let $V$ be a finite dimensional
vector space over $K$, let $d \ge 2$ be an integer, and consider an element $f$ of $S^d V^*$, i.e.,
a homogeneous form of degree $d$ on $V$. The {\em strength} of $f$
is the minimal $r$ for which there is an expression
\begin{displaymath}
f=\sum_{i=1}^r g_i h_i
\end{displaymath}
where, for every $i$, $g_i \in S^{e_i} V^*$ and $h_i \in S^{d-e_i} V^*$ for
some $e_i$ strictly between $0$ and $d$. We stress that the $e_i$ may vary with $i$. 
Strength is a complexity
measure of a homogeneous polynomial that plays a fundamental role in commutative
algebra \cite{ah}, number theory \cite{schmidt}, and algebraic combinatorics
\cite{kazi}. Let $X_{d,r}(V)$ be the strength $\leq r$ locus in $S^d V^*$.

When $d=2$, so that $f$ is a quadratic form, the strength of $f$ equals
$\lceil \rank(\beta)/2 \rceil$, where $\beta$ is the bilinear
form associated to $f$, and hence the strength of $f$ can be detected using minors. In
particular, $X_{2,r}(V)$ is a Zariski-closed subset of $S^2 V^*$ for every
$r$ and $V$. For $d=3$, strength is the same as {\em q-rank} \cite{des} or
{\em slice rank} \cite{bbov}, and $X_{3,r}(V)$ is still Zariski-closed for every $r$
and $V$ \cite{des}. Furthermore, for $r=1$ and arbitrary $d$, $X_{d,1}(V)$ is 
the variety of reducible forms of degree $d$ and also Zariski-closed. 
However, for most pairs $(d,r)$ it is expected that $X_{d,r}(V)$ is
not Zariski closed when $\dim(V)$ is sufficiently large. In particular,
this is true for $d=4$ and $r=3$ \cite{bbov2}. 

For $d,r,V$ arbitrary, let
$\ol{X}_{d,r}(V)$ be the Zariski closure of $X_{d,r}(V)$. Elements of $\ol{X}_{d,r}(V)$
are said to have \defn{border strength} $\le r$. Understanding the
varieties $\ol{X}_{d,r}(V)$ is a difficult and important problem. Our
main theorem on $\GL$-varieties yields an interesting new result about
the irreducible components of these varieties, which we now describe.

Fix a sequence $\be=(e_1,\ldots,e_r) \in
\{1,\ldots,d-1\}^r$ and consider the locus $X_{d,r,\be}(V) \subseteq
X_{d,r}(V)$ consisting of all elements of the form
\begin{displaymath} \sum_{i=1}^r g_i h_i \text{ where } g_i \in S^{e_i} V^*, h_i \in S^{d-e_i} V^*. \end{displaymath}
Clearly, $\ol{X}_{d,r}(V)$ is the union of $\ol{X}_{d,r,\be}(V)$ over all $\be \in
\{1,\ldots,d-1\}^r$, and the latter varieties are irreducible. Moreover, for $d,r$ fixed and
$\dim(V)$ sufficiently large, it is not hard to show that the $\ol{X}_{d,r,\be}(V)$ are precisely the
irreducible components of $\ol{X}_{d,r}(V)$, with $\ol{X}_{d,r,\be}(V)=\ol{X}_{d,r,\be'}(V)$ if and
only if $\be'$ and $\be$ differ by a permutation. 

\begin{theorem} \label{thm:brank2}
Given $d$ and $r$ and $\be \in \{1,\ldots,d-1\}^r$, there exists an
irreducible affine variety $B=B(d,r,\be)$ such that for any finite
dimensional vector space $V$ there is an integer $m \ge 0$ and a
surjective map of varieties $B \times \bA^m \to \ol{X}_{d,r,\be}(V)$.
\end{theorem}

We stress that $B$ is irreducible in this theorem; if this requirement is
dropped, then the result is a consequence of the unirationality theorem
in \cite{polygeom}, along with the fact that $\GL$-varieties are Noetherian \cite{draisma}. 
The theorem above implies that some aspects of the
geometry of $\ol{X}_{d,r,\be}(V)$ stabilize with $V$. In fact, our main
results imply an even more precise form of Theorem~\ref{thm:brank2}; e.g.,
the affine space $\bA^m$ and the surjections from $B \times \bA^m$ vary
in a manner compatible with $V$. For instance, $\GL(V)$ acts trivially
on $B$ and linearly on $\bA^m$, and the surjection $B \times \bA^m \to
\ol{X}_{d,r,\be}(V)$ is $\GL(V)$-equivariant.

Here is a corollary of the more precise theorem:

\begin{corollary}
Given $d$ and $r$ and $\be$, there exists an integer $D=D(d,r,\be)$ such that for any finite
dimensional vector space $V$, any two points of $\ol{X}_{d,r,\be}(V)$ can be joined by an
irreducible curve in $\ol{X}_{d,r,\be}(V)$ of degree $\le D$.
\end{corollary}

The theory that we develop below can be applied in a much broader context.
For instance, all of the results above hold, \textit{mutatis mutandis},
also for partition rank of tensors.

\subsection{$\GL$-varieties} \label{ssec:GLvar}

To state our main results, we recall the set-up from \cite{polygeom};
more background is given in \S \ref{s:bg}. Let $\GL=\bigcup_{n \ge
1} \GL_n(K)$ and let $\bV=\bigcup_{n \ge 1} K^n$. Let $\bV_{\lambda}$
be the irreducible polynomial representation of $\GL$ corresponding to
the partition $\lambda$, i.e., $\bV_{\lambda}=\bS_{\lambda}(\bV)$, where
$\bS_{\lambda}$ is the Schur functor. Let $\bA^{\lambda}$ be the affine
scheme $\Spec(\Sym(\bV_{\lambda}))$. For example, if $\lambda=(d)$ then
$\bA^{\lambda}$ is the space of degree-$d$ forms on $\bV$, a typical
$K$-point of which has the form $\sum_{i=(i_1 \leq \cdots \leq i_d)}
c_i x_{i_1} \cdots x_{i_d}$, where the $c_i$ are in $K$ (and may all
be nonzero).  For a tuple of partitions $\ulambda=[\lambda_1, \ldots,
\lambda_r]$, we let $\bA^{\ulambda}=\bA^{\lambda_1} \times \cdots \times
\bA^{\lambda_r}$.

An affine \defn{$\GL$-variety} is a closed $\GL$-stable reduced subscheme
of $\bA^{\ulambda}$. The $\bA^{\ulambda}$ themselves are the basic
examples of $\GL$-varieties; they play the role in our theory that
affine spaces play for finite-dimensional affine varieties. Any finite
dimensional variety is also a $\GL$-variety, with trivial $\GL$-action:
it lives in a space of the form $\bA^{[\emptyset, \ldots, \emptyset]}$, a product of
copies of $\bA^{\emptyset}=\bA^1$ with trivial $\GL$-action.
And above
we have already encountered a more interesting $\GL$-variety, namely,
the border strength $\leq r$ locus in $\bA^{(d)}$: this is the inverse
limit $\lim_{\leftarrow n} \ol{X}_{d,r}(K^n)$ along the maps $S^d (K^{n+1})^*
\to S^d (K^n)^*$ that sends the last variable in a degree-$d$ form in
$n+1$ variables to zero.

\subsection{Main results} \label{ss:main}

We can now state our main theorems. In what follows, we fix an irreducible $\GL$-variety $X$. 

\begin{thma}[Strong unirationality theorem] \label{thm:stronguni}
There is a surjective map of $\GL$-varieties $\phi \colon B \times \bA^{\umu} \to X$ for some irreducible affine variety $B$ and some tuple $\umu$.
\end{thma}

This theorem allows us to prove several new geometric results about $\GL$-varieties. The first, simple consequence concerns curves (see \S \ref{ss:not} for our definition of curve).

\begin{thma}[Existence of curves] \label{thm:Curves}
Let $x$ and $y$ be $K$-points of $X$. Then there exists an irreducible curve $C$ over $K$ and a map
$j \colon C \to X$ of $K$-schemes such that $x,y \in \im(j)$. 
\end{thma}

The second consequence concerns mapping spaces. Let $\umu$ be a pure tuple, meaning it does not contain the empty partition $\emptyset$, and let $\ulambda$ be any tuple. By basic representation theory, the $\GL$-equivariant morphisms $\bA^{\umu} \to \bA^{\ulambda}$ form (the $K$-points of) a finite dimensional affine space. Consequently, for a closed $\GL$-subvariety $X$ of $\bA^{\ulambda}$, the $\GL$-equivariant morphisms $\bA^{\umu} \to X$ form a closed subvariety of that affine space. This \defn{mapping space} is denoted $\cM_{\umu}(X)$, and can be characterized intrinsically via its functor of points \cite[\S 2.6]{polygeom}.

\begin{thma}[Irreducibility of mapping spaces] \label{thm:MappingSpace}
The mapping space $\cM_{\umu}(X)$ is irreducible.
\end{thma}

The final consequence concerns the orbit space of $X$. We say that two points belong to the same \defn{generalized orbit} if each belongs to the Zariski closure of the $\GL$-orbit of the other. We define the \defn{orbit space} of $X$, denoted $X^{\orb}$, to be the space of generalized orbits, equipped with the quotient topology. This is an important space to understand if one cares about $X$.

In \cite{polygeom}, we initiated an approach to describing $X^{\orb}$. First, we defined a notion of \defn{type} for a point in $X$, which is a pure tuple that is invariant on generalized orbits; see \S \ref{ss:type} for the definition. We then define $X^{\orb}_{\ulambda}$ to be the subspace of $X^{\orb}$ consisting of points whose type can be obtained from $\ulambda$ by deleting some partitions; these subspaces filter $X^{\orb}$. Second, we define a space $X^{\type}_{\ulambda}$ that we called the \defn{type space} of $X$. The key property of this space is that it is defined in terms of finite dimensional algebraic varieties. Finally, we defined a canonical map
\begin{displaymath}
\rho_{\ulambda} \colon X_{\ulambda}^{\type} \to X_{\ulambda}^{\orb},
\end{displaymath}
and showed that it is a continuous bijection \cite[Theorem~1.3]{polygeom}. Thus, at this point, we have a finitary description of $X^{\orb}_{\ulambda}$ as a set, but not as a space. We expected that $\rho_{\ulambda}$ should be a homeomorphism, but this was out of reach with previous methods. We are now able to show this, completing the finitary description of $X^{\orb}_{\ulambda}$.

\begin{thma}[Topology of generalized orbits] \label{thm:GeneralizedOrbit}
The map $\rho_{\ulambda}$ is a homeomorphism.
\end{thma}

See \S \ref{ss:type} for additional details and an example.

\begin{remark}
Throughout this paper our coefficient field $K$ is algebraically closed of characteristic zero. The first requirement can be dropped in many places, and occasionally we will make a remark to that effect. The
characteristic zero requirement is more essential: first of all, the structure of $\GL$-varieties in positive characteristic is already much more delicate, and moreover we freely use characteristic-zero techniques such as resolution of singularities. In \cite{poschar}, we studied $\GL$-varieties in positive characteristic, and established an analog of the unirationlity theorem in \cite[Theorem~6.7]{poschar}. It would be interesting if there were an analog to the present paper.
\end{remark}

\subsection{Sketch of proof} \label{ss:sketch}

We make a few comments on the proof of Theorem~\ref{thm:stronguni}. Let $X$ be an irreducible $\GL$-variety, and choose a closed embedding $X \subset \bA^{\ulambda}$. By the unirationality theorem proved in \cite{polygeom}, there is a dominant map $\phi \colon B \times \bA^{\umu} \to X$ with $B$ a finite dimensional irreducible variety. For simplicity, we assume here that $B$ is a point, and thus omit it. The proof has three main steps.

\textit{(a) Enlarging the image.} Suppose there is a $K$-point point $x \in X$ that is not in the image of $\phi$. By the results of \cite{imgclosure}, we can realize $x$ as a limit of a 1-parameter family with bounded denominators. Concretely, this means that there are $K$-points $v_{-n}, \ldots, v_m$ of $\bA^{\umu}$ such that
\begin{displaymath}
x = \lim_{t \to 0} \phi(v_{-n} t^{-n} + \cdots + t^m v_m).
\end{displaymath}
Let $s=n+m+1$, and let $\bA^{s \cdot \umu}$ denote the space $(\bA^{\umu})^s$. For $w \in \bA^{s \cdot \umu}$, let $p_w(t)$ be the Laurent polynomial $w_{-n} t^{-n}+\cdots+t^m w_m$. The limit $\phi(p_w(t)))$ as $t \to 0$ converges if and only if the coefficients of negative powers of $t$ vanish. This is a system of polynomial equations in the $w$'s. It follows that the set of $w$'s with convergent limit form a closed $\GL$-subvariety $Y$ of $\bA^{s \cdot \umu}$.

Essentially definition, $\phi$ induces a map $\psi \colon Y \to X$ by $w\mapsto \lim_{t \to 0} \phi(p_w(t)$). The image of $\psi$ contains the point $x$, since $x=\psi(v)$. It also contains the image of $\phi$, as ones sees by evaluating at constant Laurent polynomials. Thus $\im(\psi)$ contains $\im(\phi) \cup \{x\}$, and in this way, $\psi$ is ``better'' than $\phi$. See \S \ref{s:enlarge} for details. In \S \ref{s:enlarge2}, we prove a variant where $x$ can be any scheme-theoretic point. We note that dealing with the case where $B$ is non-trivial introduces a number of complications.

\textit{(b) The core argument.} Let $\Vert \unu \Vert$ be the largest size of a partition in the tuple $\unu$. Define the \defn{embedding degree} of a $\GL$-variety $Z$, denoted $\ed(Z)$, to be the minimal $d$ such that $Z$ admits a closed embedding into $\bA^{\unu}$ with $\Vert \unu \Vert=d$, and define the \defn{unirationality degree} of $Z$, denoted $\ud(Z)$, to be the minimal $d$ such that there is a dominant map $C \times \bA^{\unu} \to Z$ with $C$ finite dimensional and $\Vert \unu \Vert=d$. The unirationality theorem from \cite{polygeom} shows that $\ud(Z) \le \ed(Z)$.

Let $d=\ud(X)$, and suppose we are in the fortunate situation where we have a strict inequality $\ud(X)<d$; furthermore, suppose we have chosen $\ulambda$ and $\umu$ optimally, i.e., $\Vert \ulambda \Vert=d$ and $\Vert \umu \Vert = \ud(X)$. Assume we know Theorem~\ref{thm:stronguni} holds when the embedding degree is $<d$ (our ultimate argument is an induction on embedding degree). Since $\ed(Y) = \Vert \umu \Vert < d$, we know Theorem~\ref{thm:stronguni} holds\footnote{Our $Y$ is not necessarily irreducible. We actually replace $Y$ with a suitably chosen irreducible component.} for $Y$, and so we can find a surjective map $B' \times \bA^{\unu} \to Y$. Composing with $\psi$, we thus have a dominant map $\phi' \colon B' \times \bA^{\unu} \to X$ whose image contains $\im(\phi) \cup \{x\}$. We have thus constructed an improvement to our original $\phi$. By iterating this argument\footnote{There is one subtlety here: to iterate, we must have $\Vert \unu \Vert \le d$. We thus actually prove a stronger statement than Theorem~\ref{thm:stronguni}, which states that the tuple $\umu$ in Theorem~\ref{thm:stronguni} can be chosen to satisfy $\Vert \umu \Vert = \ud(X)$.}, and using the noetherianity of $\GL$-varieties \cite{draisma}, we eventually reach a surjective map, as in Theorem~\ref{thm:stronguni}. See Theorem~\ref{thm:uni-improved} for details.

\textit{(c) Pulling off the top piece.} The above argument only works when $\ud(X)<\ed(X)$. To complete the proof, we need one additional result (see \S \ref{s:top}). Let $d=\ed(X)$. We show that there is a surjective map $X' \times \bA^{\ul{\tau}} \to X$ where $X'$ is an irreducible $\GL$-variety with $\ud(X')<d$ and $\ed(X') \le d$, and every partition in $\tau$ has size $d$. We can then apply the argument from step (b) to prove Theorem~\ref{thm:stronguni} for $X'$. The theorem for $X$ then follows easily.

Here is the basic reason it is feasible to prove such a result. Recall that we have a dominant map $\phi \colon B \times \bA^{\umu} \to X \subset \bA^{\ulambda}$ with $\Vert \ulambda \Vert=d$ and $\Vert \umu \Vert \le d$. The key point is that if $\alpha$ and $\beta$ are two partitions of $d$ then any map $\bA^{\alpha} \to \bA^{\beta}$ of $\GL$-varieties is linear. This basically means that $\phi$ has a very simple behavior on the degree $d$ piece of $\bA^{\umu}$, which leads to the approximate factorization of $X$ given above.

\subsection{Notation} \label{ss:not}

Important notation:
\begin{description}[align=right,labelwidth=2cm,leftmargin=!]
\item [ $K$ ] the base field (characteristic~0 and algebraically closed)
\item [ $\Omega$ ] an extension of $K$
\item [ $\GL$ ] the infinite general linear group
\item [ $\ulambda$ ] a tuple of partitions $[\lambda_1, \ldots, \lambda_r]$, called \defn{pure} if no $\lambda_i$ is the empty partition
\item [ $\bA^{\ulambda}$ ] the basic affine $\GL$-variety
\item [ {$K[X]$} ] the coordinate ring of the affine scheme $X/K$
\item [ $\Vert \ulambda \Vert$ ] the maximum size of a partition in $\ulambda$ (\S \ref{ss:top})
\item [ $\ed$ ] embedding degree (\S \ref{ss:top})
\item [ $\ud$ ] unirationality degree (\S \ref{ss:top})
\end{description}
A \defn{variety} over a field $\Omega$ is a reduced scheme that is separated and of finite type over $\Omega$. Most varieties in this paper are affine. A \defn{curve} is a one-dimensional variety; we do not require curves to be smooth or irreducible, unless explicitly stated. A \defn{family of curves} over a base scheme $S$ is a morphism $C \to S$ that is flat, separated, and of finite type, and whose fibers are curves. We typically denote ordinary, finite dimensional varieties by $B$ or $C$, and $\GL$-varieties by $X$ or $Y$.

\subsection{Acknowledgements}

We thank Bhargav Bhatt for helpful conversations.

\section{Background on \texorpdfstring{$\GL$}{GL}-varieties} \label{s:bg}

In this section, we recall the basic definitions surrounding $\GL$-varieties, and also prove some simple new results relevant to the current paper. We refer to \cite{polygeom} and \cite{imgclosure} for detailed background material on $\GL$-varieties.

\subsection{$\GL$-varieties}

A \defn{$\GL$-algebra} $R$ is a commutative $K$-algebra equip\-ped with
an action of $\GL$ via automorphisms that turns it into a polynomial
representation of $\GL$, i.e., a subquotient of a (typically infinite)
direct sum of tensor powers of $\bV$. By standard representation
theory, it is then a direct sum of representations of the form
$\bS_{\lambda}(\bV)$. We say that $R$ is \defn{$\GL$-finitely generated}
if it is generated by the $\GL$-orbits of finitely many elements; it
is then a quotient of $\GL$-algebra of the form $\bigotimes_{i=1}^r
\Sym(\bS_{\lambda_i}(\bV))$ for some tuple $\ulambda$. An
\defn{affine $\GL$-variety} is the spectrum of a reduced, $\GL$-finitely
generated $\GL$-algebra. Equivalently, it is a $\GL$-stable closed
and reduced subscheme of some $\bA^{\ulambda}$. A \defn{quasi-affine
$\GL$-variety} is $\GL$-stable open subscheme of some $\GL$-variety.
A \defn{morphism} (or, informally, just map) of $\GL$-varieties is a
$\GL$-equivariant morphism of schemes.

\subsection{$\GL$-varieties as functors}

Let $M$ be a polynomial representation of $\GL$. Then $M$ admits a canonical direct sum decomposition $\bigoplus_{\lambda} M_{\lambda} \otimes \bS_{\lambda}(\bV)$ where the sum is over partitions and $M_{\lambda}$ is a multiplicitiy space. For a vector space $V$, we define $M\{V\}$ to be $\bigoplus_{\lambda} M_{\lambda} \otimes \bS_{\lambda}(V)$. In this way, $V \mapsto M\{V\}$ is a functor of the vector space $V$. We use curly braces here to avoid confusion with functor of points in the geometric context.

Suppose $A$ is a $\GL$-algebra. Then $V \mapsto A\{V\}$ is a functor to the category of $K$-algebras. In particular, the natural inclusion $K^n \to \bV$ gives an inclusion of algebras $A\{K^n\} \to A\{V\}$, while the natural surjection $\bV \to K^n$ yields a surjection of algebras $A \to A\{K^n\}$. If $A$ is $\GL$-finitely generated then each $A\{K^n\}$ is finitely generated.

Dually, an affine $\GL$-variety $X$ gives a contravariant functor from
finite-dimensional vector spaces to varieties; we write $X\{V\}$ for the
evaluation of that functor at $V$. We have morphisms of schemes from $X$
to $X\{K^n\}$ and vice versa, and the former is a left inverse of the
latter. We note that a \emph{quasi-affine} $\GL$-variety does not
yield such a functor.

To derive Theorem~\ref{thm:brank2} from
Theorem~\ref{thm:stronguni} we need the following proposition.

\begin{proposition} \label{prop:functorial}
Let $\phi \colon Y \to X$ be a morphism of affine
$\GL$-varieties, and let $Z=\ol{\im(\phi)}$. Then for any
finite dimensional vector space $V$, the variety $Z\{V\}$ is
the image closure of $\phi\{V\} \colon Y\{V\} \to X\{V\}$.
\end{proposition}

\begin{proof}
It suffices to check this for the case where $V=K^n$. The 
commutative diagram
\begin{displaymath}
\xymatrix{
X \ar[r]^\phi \ar[d] & Y \ar[d] \\
X\{K^n\} \ar[r]_{\phi\{K^n\}} & Y\{K^n\}
}
\end{displaymath}
implies that the right-most downward arrow maps $Z$ into 
the image closure of $\phi\{K^n\}$, so that $Z\{K^n\}
\subseteq \ol{\im(\phi\{K^n\})}$. Conversely, 
the commutative diagram
\begin{displaymath}
\xymatrix{
X \ar[r]^\phi & Y \\
X\{K^n\} \ar[u] \ar[r]_{\phi\{K^n\}} & Y\{K^n\} \ar[u]
}
\end{displaymath}
shows that the image closure of $\phi\{K^n\}$ is mapped into
$Z$ via the right-most upward arrow, and hence contained in
$Z\{K^n\}$. 
\end{proof}

\begin{proof}[Proof of Theorem~\ref{thm:brank2} from Theorem~\ref{thm:stronguni}]
Set $X:=\lim_{\leftarrow n} \ol{X}_{d,r,\be}(K^n) \subseteq \bA^{(d)}$. By
Theorem~\ref{thm:stronguni} there exist an irreducible affine variety
$B$, a tuple $\ulambda=[\lambda_1,\ldots,\lambda_r]$, and a surjective
morphism $\phi \colon B \times \bA^{\ulambda} \to X$ of $\GL$-varieties. By
Proposition~\ref{prop:functorial}, the morphism $\phi\{V\}$ from the
finite-dimensional variety $(B \times \bA^{\ulambda})\{V\}=B \times
\prod_{i=1}^r \Sym(\bS_{\lambda_i}(V)))$ to $X$ is surjective. The domain
of $\phi\{V\}$ is of the form $B \times \bA^m$, where $m$ depends on $V$
but $B$ does not, as desired.
\end{proof}

We note that the proof gives much more: $\phi\{V\}$ depends in a
functorial manner on $V$ and is, in particular, $\GL(V)$-equivariant.

\subsection{The central grading}

A $\GL$-algebra $R$ has a natural $\bZ_{\geq 0}$-grading, called the \defn{central
grading}, in which $f \in R$ has degree $d$ if $(t \cdot \id_n)f=t^d f$
for all $t \in K^*$ and all $n \gg 0$. In this grading, the elements of
$\bV_\lambda$ have degree $|\lambda|$. If $X$ is an affine $\GL$-variety,
then the central grading of its coordinate ring $K[X]$ gives rise to
an action of the multiplicative group $\bG_m$ on $X$, called the
\defn{central $\bG_m$-action}. Every morphism of affine $\GL$-varieties
is equivariant with respect to the central $\bG_m$-action. We write
$X_0$ for the spectrum of the degree-$0$ part of $K[X]$; this is a
variety isomorphic to $X\{K^0\}$.

We write $K\lpp t \rpp$ for the field of Laurent series over $K$.
The central action on $X$
yields an action of $(\bZ,+)$ on $K \lpp t \rpp$-valued points
of $X$. Indeed, a $K\lpp t \rpp$-point of $X$ corresponds to a
$K$-algebra homomorphism $f \colon K[X] \to K \lpp t \rpp$, and we let $n
\in \bZ$ send $f$ to the homomorphism, denoted $t^n \star f$, that maps
a homogeneous element $r \in K[X]$ of degree $d$ in the central grading
to $t^{dn} f(r)$. Note that $t^m \star (t^n \star f)=t^{m+n} \star f$
as desired.

\subsection{Mapping spaces} \label{ssec:MappingSpaces}

Let $X$ be an affine $\GL$-variety and let $X_0$ be the spectrum of the
degree-$0$ part of the coordinate ring $K[X]$. Let $\ulambda$ be a pure
tuple. In \cite{polygeom} we constructed the mapping space $\cM_{\ulambda}(X)$
that parametrizes morphisms $\bA^{\ulambda} \to X$ of $\GL$-varieties. On
$\cM_{\ulambda}(X)$ acts the multiplicative group $\bG_m$ via its central
action on $X$.

We show that the central $\bG_m$-action
also yields an action of $(\bZ,+)$ on $K \lpp t \rpp$-valued
points of $\cM_{\ulambda}(X)$. Indeed, a $K\lpp t \rpp$-point
$\gamma$ of $\cM_{\ulambda}(X)$ corresponds to a $\GL$-equivariant
$K$-algebra homomorphism $f \colon K[X] \to K\lpp t \rpp \otimes
\Sym(\bV_{\ulambda})$. An $n \in \bZ$ sends this $f$ to the homomorphism
that maps a degree-$d$ element $r \in K[X]$ to $t^{nd} f(r)$. We write
$t^n \star \gamma$ for the corresponding $K\lpp t \rpp$-point of
$\cM_{\ulambda}$. Let $\phi \colon \cM_{\ulambda}(X) \times \bA^{\ulambda}
\to X$ be the natural map. For $K \lpp t \rpp$-points $\gamma$
of $\cM_{\ulambda}(X)$ and $\delta$ of $\bA^{\ulambda}$ and for $n
\in \bZ$ we then have
\begin{displaymath} \phi(t^n \star \gamma, t^{-n} \star \delta) =
\phi(\gamma,\delta) \in X(K \lpp t \rpp). \end{displaymath}
A straightforward check shows that if $B \subseteq \cM_{\ulambda}(X)$
is a $\bG_m$-stable subvariety, then the $\bZ$-action on $K \lpp t
\rpp$-points of $\cM_{\ulambda}(X)$ preserves $B(K \lpp
t \rpp)$. This, and the following proposition, will be used in \S\ref{ssec:LimB}.

\begin{proposition} \label{prop:map-limit}
Let $\gamma$ be a $K \lpp t \rpp$-point of $\cM_{\ulambda}(X)$ such
that the image of $\gamma$ in $X_0$ converges. Then $t^n \star \gamma$
converges in $\cM_{\ulambda}(X)$ for some $n$.
\end{proposition}

\begin{proof} 
As above, the point $\gamma$ corresponds to a $\GL$-equivariant algebra
homomorphism $f \colon K[X] \to K\lpp t \rpp \otimes_K \Sym(\bV_{\ulambda})$. The
assumption that the image of $\gamma$ in $X_0$ converges means that 
$f$ maps the degree-$0$ part $K[X_0]$ of $K[X]$ into (the degree-$0$
part of) $K\lbb t\rbb \otimes_K \Sym(\bV_{\ulambda})$. 
The point $t^n \star \gamma$ corresponds to the algebra
homomorphism $\tilde{f}$ that maps a homogeneous $r \in K[X]$ of degree $d$ 
to $t^{nd} f(r)$. Now $K[X]$ is $\GL$-finitely generated,
say by $K[X_0]$ and homogeneous elements $r_1,\ldots,r_k$ of degrees
$d_1,\ldots,d_k>0$. Let $e_i$ be the minimal exponent of
$t$ appearing in $f(r_i)$. Then the minimal exponent of $t$
in $\tilde{f}(r_i)$ is $e_i + d_i n$. Take $n$ such that
these numbers are all nonnegative. Then $t^n \star \gamma$ is a 
$K\lbb t \rbb$-point of $\cM_{\ulambda}$, as desired. 
\end{proof}

\section{Geometric construction of limits} \label{s:geolim}

\subsection{Overview}

Let $X$ be an affine $\GL$-variety, embedded as a $\GL$-stable closed
subscheme of $\bA^{\ulambda}$. A $K\lpp t \rpp$-point $\gamma(t)$ of $X$
is called \defn{bounded} if there exists an integer $N \geq 0$ such that
$t^N \cdot \gamma(t)$ is a $K\lbb t \rbb$-point of $\bA^{\ulambda}$.
Informally, only finitely many negative exponents of $t$ appear in the
(typically infinitely many) coordinates of $\gamma(t)$. It is not hard to
see that this notion is independent of the choice of a closed embedding
of $X$ into one of the basic $\GL$-schemes, see \cite[\S6.1]{imgclosure}.

Let $\umu$ and $\ulambda$ be tuples, and let $\phi \colon \bA^{\umu}
\to \bA^{\ulambda}$ be a map of $\GL$-varieties. By \cite[Theorem
6.6]{imgclosure} and our assumption in the current paper that $K$ is
algebraically closed, every point of $\ol{\im{\phi}}$ can be realized in
the form $\lim_{t \to 0} \phi(y(t))$, where $y(t)$ is a bounded $K\lpp
t \rpp$-point of $\bA^{\umu}$. It is not difficult to see that one can
in fact take $y(t)$ to be a Laurent polynomial instead of a Laurent
series, as very high powers of $t$ do not affect the limit.

Now, fix nonnegative integers $n$ and $m$, and consider $y(t)$
of the form $\sum_{i=-n}^m t^i v_i$. Such a $y(t)$ is parametrized by
$(v_{-n}, \ldots, v_m)$. We thus see that the space of such $y(t)$'s
is the $\GL$-variety $\bA^{(n+m+1) \cdot \umu}$. As we will see, the
$y(t)$'s for which $\lim_{t \to 0} \phi(y(t))$ exists forms a closed
$\GL$-subvariety $Y$, and the map $\phi$ induces a map $\psi \colon Y \to \bA^{\ulambda}$.
This is a very useful perspective, since it allows us to realize points in the image
closure of $\phi$ as points in the image of $\psi$ (see \S \ref{ss:sketch}). In this section, we develop this idea in detail,
first in the setting above, and then in the more difficult case where
the domain of $\phi$ is of the form $B \times \bA^{\umu}$ where $B$
is a finite dimensional affine variety.

\subsection{First construction}

Let $\phi \colon \bA^{\umu} \to \bA^{\ulambda}$ be a map of
$\GL$-varieties. Then there exists a natural number $d$ such that 
the coordinates of $\phi(x)$ are polynomials of total degree $\le d$
in the coordinates of $x$; we say that $\phi$ has degree $\leq d$. Let $L_{n,m} = \bigoplus_{i=-n}^m K t^i$ be the space of Laurent polynomials involving $t^{-n}, \ldots, t^m$. We let $L_{n,m} \otimes \bA^{\umu}$ denote the $\GL$-variety $\bA^{(n+m+1) \cdot \umu}$; a point of this space can be written as $\sum_{i=-n}^m v_i t^i$ with $v_i \in \bA^{\umu}$. The map $\phi$ induces a map of $\GL$-varieties
\begin{displaymath}
\phi \colon L_{n,m} \otimes \bA^{\umu} \to L_{dn,dm} \otimes \bA^{\ulambda},
\end{displaymath}
by simply plugging the Laurent polynomials into $\phi$. We define $\fL_{n,m}=\fL_{n,m}(\phi)$ to be the inverse image of $L_{0,dm} \otimes \bA^{\ulambda}$ under the above map. This is clearly a closed $\GL$-subvariety of $L_{n,m} \otimes \bA^{\umu}$. Explicitly, $\sum_{i=-n}^m v_i t^i$ belongs to $\fL_{n,m}$ if $\phi(\sum_{i=-n}^m v_i t^i)$ has no negative powers of $t$. We have a natural map of $\GL$-varieties
\begin{displaymath}
\psi \colon \fL_{n,m} \to \bA^{\ulambda}
\end{displaymath}
given by applying $\phi$ and evaluating at $t=0$; the key point here
is that there are no negative powers of $t$, so this is a well-defined
operation. Before going any further with the general theory, we look at an example:

\begin{example}
Let $\umu=[(2),(2),(2)]$, let $\ulambda=[(4)]$, and consider the map
\begin{displaymath}
\phi \colon \bA^{\umu} \to \bA^{\ulambda}, \qquad
\phi(f,g,h)=fg-h^2,
\end{displaymath}
where here $f$, $g$, and $h$ are points in $\bA^{(2)}$. We consider the induced map on $L_{-1,m}$ with $m \ge 1$. Write $f=\sum_{i=-1}^m f_i t^i$, and similarly for $g$ and $h$. We have
\begin{align*}
\phi(f,g,h) =
& \tfrac{1}{t^2}(f_{-1}g_{-1}-h_{-1}^2)
+\tfrac{1}{t}(f_{-1}g_0+f_0g_{-1}-2h_{-1}h_0) \\
&+ (f_{-1}g_1+f_0g_0+f_1g_{-1}-2h_{-1}h_1+h_0^2)+O(t)
\end{align*}
where $O(t)$ means all remaining terms are divisible by $t$. We thus see that $\fL_{-1,m}$ is defined by the equations
\begin{displaymath}
f_{-1}g_{-1}-h_{-1}^2 = 0, \qquad
f_{-1}g_0+f_0g_{-1}-2h_{-1}h_0 = 0,
\end{displaymath}
and the map $\psi\colon \fL_{-1,m} \to \bA^{\ulambda}$ is given by
\begin{displaymath}
(f,g,h) \mapsto f_{-1}g_1+f_0g_0+f_1g_{-1}-2h_{-1}h_1+h_0^2.
\end{displaymath}
Notice that this is essentially independent of $m$ once $m \ge 1$.
\end{example}

The following proposition gives the most important properties of this construction.

\begin{proposition} \label{prop:geolim}
Let $\phi \colon \bA^{\umu} \to \bA^{\ulambda}$ have degree $\le d$.
\begin{enumerate}
\item Let $n \ge 0$ be given, let $m_0=n(d-1)+1$, and let $m \ge m_0$. Given $f,g \in L_{n,m} \otimes \bA^{\umu}$ with $f-g=O(t^{m_0})$, we have $\phi(f)-\phi(g)=O(t)$.
\item Let $n' \ge n$. Then $\fL_{n,m}$ is a closed $\GL$-subvariety of $\fL_{n',m}$.
\item We have $\psi(\fL_{n,m}) \subset \ol{\im{\phi}}$ for all $n$ and $m$.
\item There exists $n$ and $m$ such that we have equality in (c).
\end{enumerate}
\end{proposition}

\begin{proof}
(a), (b), and (c) are easy. For (d), let $Y_n$ be the image of
$\fL_{n,m}$ for $m \gg 0$; this is independent of $m$ by (a). By 
Chevalley's theorem for $\GL$-varieties \cite[Theorem 7.13]{polygeom}, the
$Y_n$'s form an ascending chain of $\GL$-constructible subsets of
$\bA^{\umu}$. By \cite[Theorem 6.6]{imgclosure}, we have
$\ol{\im{\phi}}=\bigcup_{n \ge 1} Y_n$. 
But this implies $\ol{\im{\phi}}=Y_n$ for some $n$, which completes the proof.
\end{proof}

The following is essentially \cite[Theorem 1.5]{imgclosure}

\begin{corollary} \label{cor:boundeddenominator}
Let $\phi \colon \bA^{\umu} \to \bA^{\ulambda}$ be given.
Then there exists an integer $n \ge 0$ with the following
property: if $x$ is a $K$-point of $\ol{\im{\phi}}$,
then $x$ can be realized as $\lim_{t \to 0} \phi(y(t))$,
where $y(t)$ is a $K\lpp t \rpp$-point of $\bA^{\umu}$
such that $t^n y(t)$ has no negative exponents of $t$. 
\end{corollary}

\begin{proof}
Let $n$ be such that $\psi(\fL_{n,m})=\ol{\im{\phi}}$ (for
$m \gg 0$), which exists by the proposition. Thus $\psi
\colon \fL_{n,m} \to \ol{\im{\phi}}$ is a surjection of
$\GL$-varieties. Given $x$ as in the corollary statement, we
can find a pre-image of $x$ in $\fL_{n,m}$ defined over
$K$. This gives the desired point $y(t)$. 
\end{proof}

\begin{remark} \label{rmk:global}
The map $\fL_{n,m} \to \bA^{\ulambda}$ is obtained by evaluating the relevant (Laurent) polynomial at~0. In fact, one can evaluate at non-zero values of $t$ as well. We therefore have a map
\begin{displaymath}
\psi \colon \fL_{n,m} \times \bA^1 \to \bA^{\ulambda}.
\end{displaymath}
We prove this in detail below, in a more general situation.
\end{remark}

\subsection{Second construction}

Suppose now we have a map $\phi \colon B \times \bA^{\umu} \to
\bA^{\ulambda}$, where $B$ is an affine variety. We now examine limits
where we also vary in the $B$ direction. Fix an infinitesimal curve
$\gamma \colon \Spec(K \lbb t \rbb) \to B$. For our purposes, it will suffice to consider the space of limits of the form $\phi(\gamma(t), y(t))$, where $y$ varies.

The basic definitions are similar to before. We have a natural map
\begin{displaymath}
L_{n,m} \otimes \bA^{\umu} \to L_{dn,dm} \otimes \bA^{\ulambda}, \qquad
y(t) \mapsto \phi(\gamma(t), y(t)),
\end{displaymath}
where we simply discard terms past $t^{dm}$. We define $\fL_{n,m}^{\gamma}=\fL_{n,m}^{\gamma}(\phi)$ to be the inverse image of $L_{0,dm} \otimes \bA^{\ulambda}$. There is once again a map
\begin{displaymath}
\psi \colon \fL_{n,m}^{\gamma} \to \bA^{\ulambda}.
\end{displaymath}
It has similar properties to the previous case, with similar proofs.

We are now interested in ``globalizing'' this construction, i.e.,
evaluating away from $t=0$. We discussed this in Remark~\ref{rmk:global}
in the previous iteration of the construction. In the present case, we
must choose a curve in $B$ that is sufficiently tangent to $\gamma$;
we will then be able to define a map on all of the curve. 
For the next proposition, recall that an $m$-jet in $B$ is a
$K \lbb t \rbb / (t^{m+1})$-valued point of $B$.

\begin{proposition} \label{prop:Extendalpha}
Fix $n,m$ with $m$ sufficiently large. Let $C$ be an affine
curve over $K$, let $P$ be a smooth $K$-point of $C$, and
let $t \in K[C]$ be a uniformizer at $P$ that has no other
zero on $C$. Let $\rho \colon C \to B$ be a map such that
the $m$-jet of $B$ defined by $\rho$ with respect to $t$
agrees with $\gamma$. Let $y(t) \in \fL_{n,m}^{\gamma}$.
Then there is a morphism of $K$-schemes
\begin{displaymath}
\alpha \colon C \to \bA^{\ulambda}, \qquad \alpha(c) = \phi(\rho(c), y(t)).
\end{displaymath}
Precisely, for $c \in C \setminus P$, the expression $y(t)$
is $y(t(c))$, i.e., evaluate the function $t$ at $c$
and plug this value into the Laurent polynomial $y(t)$. At
$P$, we have $\alpha(P)=\psi(y(t))$.
\end{proposition}

\begin{proof}
Since $P$ is the only zero of $t$ on $C$, the map $\alpha$ is well-defined
on $C \setminus P$. We only need to verify that this map extends
to a morphism on all of $C$ by setting $\alpha(P):=\psi(y(t))$. Let
$\delta(t)$ be the $K \lbb t \rbb$-point of $C$ defined by $P$ and $t$.
Then it suffices to verify that $\lim_{t \to 0} \alpha(\delta(t))$
exists and equals $\psi(y(t))$. Now
\begin{displaymath} \alpha(\delta(t))=\phi(\rho(\delta(t)),y(t)) \end{displaymath}
and since, by assumption, $\rho(\delta(t))$ and $\gamma(t)$
define the same $m$-jet of $B$, and since $m$ is
sufficiently large, we have 
\begin{displaymath} \lim_{t \to 0} \phi(\rho(\delta(t)),y(t))
= \lim_{t \to 0} \phi(\gamma(t),y(t)) = \psi(t), \end{displaymath}
as desired. 
\end{proof}

Let $C \to S$ be a family of curves (see \S \ref{ss:not}) and let $P \colon S \to C$ be a
smooth point, i.e., a morphism $S \to C$ that maps each $K$-point $s
\in S$ to a smooth $K$-point of the fiber $C_s$ over $s$.  Recall that
then the image of $P$ is a relative effective Cartier divisor on $C$;
we will denote this image also by $P$. A uniformizer for $C$ at $P$ is a
regular function that generates the ideal of $P$, in a neighborhood of $P$. In
general, these only exist locally on $C$;
in the next proposition we assume that a global uniformizer exists.

\begin{proposition} \label{prop:glob-lim}
Fix $n,m$ with $m$ sufficiently large. Let $S$ be an affine
variety over $K$, let $C \to S$ be a family of affine
curves, let $P \colon S \to C$ be a smooth point of $C$
and let $t \in K[C]$ be a global uniformizer at $P$ having no other zeros. Let $\rho \colon C \to B$ be a map of varieties such that for each $K$-point $s \in S$ the $m$-jet defined by $\rho_s \colon C_s \to B$ with respect to $t \vert_{C_s}$ coincides with $\gamma$. Then there is a unique map of $\GL$-varieties
\begin{displaymath}
\alpha \colon C \times \fL_{n,m}^{\gamma} \to \bA^{\ulambda}
\end{displaymath}
that agrees with the map in Proposition~\ref{prop:Extendalpha} on each fiber above $S \times \fL_{n,m}^{\gamma}$.
\end{proposition}

\begin{proof}
There is certainly a map as above defined on $(C \setminus P)
\times \fL_{n,m}^{\gamma}$. Pick a coordinate on $\bA^{\ulambda}$.
This component of the map is a global function $f$ on $(C \setminus
P) \times \fL_{n,m}^{\gamma}$. We can thus write $f$ in the form
$\sum_{i=1}^N a_i t^{-i} + b$, where $b$ is a global function on $C \times
\fL_{n,m}^{\gamma}$, and the $a_i$'s are global functions on $S \times
\fL_{n,m}^{\gamma}$. This is just the expansion of $f$ near $P$: for some
$N \geq 0$, $t^N \cdot f$ lies in $K[C \times \fL_{n,m}^\gamma]$, and we
let $a_N \in K[S \times \fL_{n,m}^\gamma]$ be the image of $t^N \cdot f$ 
under the homomorphism dual to $P\colon S \to C$, and then regard $a_N$ as an
element of $K[C \times \fL_{n,m}^\gamma]$ via the homomorphism dual to
$C \to S$. Then $t^N \cdot f - a_N$ vanishes identically on $P$ and hence lies in the ideal generated by
$t$, hence $\tilde{f}:=f-t^{-N}a_N$ has the property that $t^{N-1}
\cdot \tilde{f}$ lies in $K[C \times \fL_{n,m}^\gamma]$, etc.
If $(s,x)$ is a $K$-point of $S \times \fL_{n,m}^{\gamma}$
then by Proposition~\ref{prop:Extendalpha}, the function extends over $P$
on the $(s,x)$ fiber. This tells us that $a_i(s,x)=0$ for all $1 \le i
\le N$. Thus $a_i=0$ identically, which shows that we have a function
on all of $C \times \fL_{n,m}^{\gamma}$.
\end{proof}

\section{Enlarging the image} \label{s:enlarge}

\subsection{Statement of results}

The purpose of \S \ref{s:enlarge} is to prove the following theorem, which is the first step of the proof sketch given in \S \ref{ss:sketch}.

\begin{theorem} \label{thm:extend}
Let $X$ be an irreducible $\GL$-variety, let $\umu$ be a pure tuple,
let $\phi \colon B
\times \bA^{\umu} \to X$ be a dominant morphism, and let $x$
be a $K$-point of $X$. Then there exists an
irreducible closed $\GL$-subvariety $Y$ of $\bA^r \times
\bA^{s \cdot \umu}$, for some $r$ and $s$, and a morphism $\psi \colon Y \to X$ of $\GL$-varieties such that $\im(\phi) \cup \{x\} \subset \im(\psi)$.
\end{theorem}

\subsection{Limits in the $B$ direction} \label{ssec:LimB}

The following proposition helps us deal with limits as we move in the $B$ direction.

\begin{proposition} \label{prop:mapping}
Let $X$ be an irreducible affine $\GL$-variety, let $\umu$ be a pure tuple, and let $\phi \colon B \times \bA^{\umu} \to X$ be a dominant morphism with $B$ irreducible. Then there exists an irreducible smooth affine variety $B'$ and a dominant morphism $\phi' \colon B' \times \bA^{\umu} \to X$ such that:
\begin{enumerate}
\item We have $\im(\phi) \subset \im(\phi')$.
\item Given $x \in X$, we can write $x=\lim_{t \to 0} \phi'(\gamma(t),
v(t))$ where $\gamma(t)$ is a $K \lbb t \rbb$-point of $B'$, and $v(t)$ is a 
bounded $K \lpp t \rpp$-point of $\bA^{\umu}$.
\end{enumerate}
\end{proposition}

The key point here is that the $\gamma$ in (b) is actually a $K
\lbb t \rbb$-point, and not just a $K\lpp t \rpp$-point. In other words, we can realize $x$ as the limit of a 1-parameter family in $B' \times \bA^{\umu}$ that already converges in the $B'$ coordinate.

\begin{proof}
Recall from \S\ref{ssec:MappingSpaces} that the mapping
space $\cM_{\umu}(X)$ is an affine variety equip\-ped with a 
$\bG_m$-action coming from the central $\bG_m$-action on
$X$.  The map $\phi$ corresponds to a map $B \to \cM_{\umu}(X)$. Now:
\begin{itemize}
\item Let $B_1 \subset \cM_{\umu}(X)$ be the image of $B$ in $\cM_{\umu}(X)$.
\item Let $B_2$ be the closure of $\bG_m \cdot B_1$.
\item Let $\pi \colon B_3 \to B_2$ be a resolution of singularities.
\item Let $\sigma \colon B_4 \to B_3$ be an affine space bundle with $B_4$ affine (Jouanolou's trick).
\end{itemize}
Let $\phi_i \colon B_i \times \bA^{\umu} \to X$ be the
natural map. We take $B'=B_4$ and $\phi'=\phi_4$. Note that
$B'$ is indeed an irreducible smooth affine variety. Since $B$ maps into $B_2$ and $B'$ surjects on $B_2$, condition (a) is clear. 

We now verify (b). Let $x \in X$ be given. By~\cite[Theorem 6.6]{imgclosure}, we can write $x=\lim_{t \to 0}
\phi_2(\gamma_2(t), v(t))$ where $\gamma_2(t)$ is a $K
\lpp t \rpp$-point of $B_2$ and $v(t)$ is a bounded $K
\lpp t \rpp$-point of $\bA^{\umu}$. For any $n \in \bZ$, we
have $t^n \star \gamma_2(t) \in B_2(K \lpp t \rpp)$ and 
\begin{displaymath}
\phi_2(\gamma_2(t), v(t))=\phi_2(t^n \star \gamma_2(t),
t^{-n} \star v(t));
\end{displaymath}
see \S\ref{ssec:MappingSpaces}. Choose $n$ such that $t^n
\star \gamma_2(t)$ has a limit in $B_2$
(Proposition~\ref{prop:map-limit}). Note that $t^{-n} \star v(t)$ is
still bounded. Relabeling, we simply assume that $\gamma_2(t)$ is a $K\lbb t \rbb$-point of $B_2$.

Since $\pi$ is surjective, $\gamma_2$ lifts to a $K\lpp t^{1/m}
\rpp$-point $\gamma_3$ of $B_3$ for some $m \ge 1$. Since $\lim_{t \to
0} \pi(\gamma_3(t))$ exists and $\pi$ is proper, it follows that
$\lim_{t \to 0} \gamma_3(t)$ exists. More precisely, the valuative
criterion for properness shows that $\gamma_3$ extends uniquely to a
$K \lbb t^{1/m} \rbb$-point of $B_3$, which we still denote by $\gamma_3$. Changing $t$ to $t^m$, we assume $m=1$. (Note that after this change of variables $v(t)$ is still bounded.)

Finally, since $\sigma$ is smooth, $\gamma_3$ lifts to a $K \lbb t \rbb$-point $\gamma_4$ of $B_4$. We have
\begin{displaymath}
\lim_{t \to 0} \phi_4(\gamma_4(t), v(t))=\lim_{t \to 0} \phi_2(\gamma_2(t), v(t)) = x.
\end{displaymath}
This completes the proof.
\end{proof}

\subsection{Construction of curves} \label{ss:jets}

We now construct a family of curves in $B$ that realizes a fixed jet at a specified point, and hits every other point of $B$.

\begin{proposition} \label{prop:jets}
Let $B$ be an irreducible smooth affine variety, let $m \ge 0$ be an integer, and let $\gamma \colon \Spec(k\lbb t \rbb/(t^{m+1})) \to B$ be an $m$-jet. We can find:
\begin{itemize}
\item An irreducible affine variety $S$.
\item A flat family of curves $C \to S$, with $C$ irreducible and affine.
\item A smooth point $P \colon S \to C$.
\item A global uniformizer $t \in K[C]$, having no zeros away from $P$.
\item A map of varieties $\rho \colon C \to B$.
\end{itemize}
such that:
\begin{enumerate}
\item For $s \in S(K)$, the $m$-jet defined by $\rho_s \colon C_s \to B$ with respect to $t$ agrees with $\gamma$.
\item The map $\rho \colon C \setminus P \to B$ is surjective.
\end{enumerate}
\end{proposition}

\begin{proof}
Let $b_1=\gamma(0)$ be the base point of the given jet and choose a finite
surjective map of varieties $\pi \colon B \to \bA^d$ such that $\pi$ is
\'etale at all points in $\pi^{-1}(\pi(b_1))$. This can be constructed as
follows. Set $d:=\dim(B)$ and assume that $B$ is a closed subvariety of
$\bA^k$ with $k$ minimal. If $k=d$, there is nothing to prove. Otherwise,
consider the incidence variety
\begin{displaymath}
Z:=\{(b,\pi) \mid b \in B \setminus\{b_1\}, 
\pi \colon \bA^k \to \bA^d \text{ linear, } \pi(b)=\pi(b_1), \text{ and }
d_b\pi|_{B} \text{ not invertible} \}.
\end{displaymath}
We claim that $\dim(Z)<k\cdot d$. Indeed, for $b \in B \setminus \{b_1\}$
sufficiently general, the vector $b-b_1 \in \bA^k$ is not in the tangent
space $T_b B$ (or else it would follow that $T_{b_1}B$ is the whole space,
which contradicts the smoothness of $B$). The fiber in $Z$ over such $b$
has codimension at least $d+1$: indeed, $\pi$ needs to have $b-b_1$ in its kernel
and these $d$ linear conditions are independent from the condition that
the restriction of $\pi$ to $T_b B$ has rank $<d$. Consequently, $Z$
has dimension $<kd$ and hence any sufficiently general $\pi\colon \bA^k \to
\bA^d$ has the property that $d_{\pi|_B}(b)$ is invertible for all $b
\in \pi^{-1}(\pi(b_1))$; here the condition for $b=b_1$ is just one more
open condition.

By ``miracle flatness'' \stacks{00R4}, $\pi$ is flat. For simplicity, we assume that  $\pi(b_1)=0$ and we write 
\begin{displaymath}
\pi(\gamma(t))=c_1(t)e_1 + \cdots + c_d(t)e_d \mod t^{m+1}
\end{displaymath}
where the $c_i$ are polynomials of degree $\leq m$ in $t$ with
$c_i(0)=0$. We further let $b_2, \ldots, b_n$ be the other points in $\pi^{-1}(0)$. Set $S:=\bA^d$. Define
\begin{displaymath}
\sigma \colon \bA^1 \times S \to \bA^d, \qquad
\sigma(\tau, v) = c_1(\tau) e_1 + \cdots + c_d(\tau) e_d + \tau^{m+1}
v.
\end{displaymath}
Note that for any fixed $K$-point $s$ of $S$, the $m$-jet defined by
$\sigma(.,s)$ at $0 \in \bA^1$ is $\pi \circ \gamma$. Let $C_0$ be the
fiber product of $\bA^1 \times S$ with $B$ over $\bA^d$, so we have
the following cartesian diagram:
\begin{displaymath}
\xymatrix{
C_0 \ar[r] \ar[d] & B \ar[d]^{\pi} \\
\bA^1 \times S \ar[r]_{\sigma} & \bA^d.
}
\end{displaymath}
Thus a $K$-point of
$C_0$ is a triple $(b, \tau, v)$ with $b \in B$ such
that 
\begin{equation} \label{eq:pisigmatau} 
\pi(b)=\sigma(\tau,0)+\tau^{m+1}v. 
\end{equation}
Note that $C_0$ is finite and flat over $\bA^1 \times S$. 

We have sections $P_i \colon S \to C_0$ by $P_i(v)=(b_i, 0, v)$ for
$1 \le i \le n$. Each of these is smooth since $\pi$ is \'etale at
$b_1,\ldots,b_n$. Specifically, fix a $K$-point $v$ of $S$. Then the
fiber $(C_0)_v$ is the preimage of $\bA^1 \times \{v\}$ in $C_0$, hence
a curve by finiteness of the map $C_0 \to \bA^1 \times S$. Furthermore, the tangent
space of $(C_0)_v$ at $P_i(v)$ consists of all triples $(x,\tau,0)
\in T_{b_i}B \times T_{0} \bA^1 \times T_v{\bA^{d}}$ such that
\begin{displaymath}
\pi(b_i+\epsilon x)=\sigma(0+\epsilon \tau,v) = \sigma(\epsilon
\tau,0) + (\epsilon \tau)^{m+1} v \mod \epsilon^2.
\end{displaymath}
The left-hand side equals $0+ \epsilon d_{\pi|_B}(b_i)x$, and since
$d_{\pi|_B}(b_i)$ is invertible, there is a unique solution $x$ for any
given $\tau$. So $T_{P_i(v)} (C_0)_v$ is one-dimensional as desired,
and it projects surjectively to $T_0 \bA^1$.

We claim that $C_0$ is
also irreducible. Indeed, its open subset where $\tau$ is nonzero is
irreducible, because because $B$ and $\bA^{1} \setminus \{0\}$ are
irreducible and one can solve \eqref{eq:pisigmatau} for $v$.
Furthermore, the locus $D$ in $C_0$ where $\tau=0$ equals $\bigcup_i
\im(P_i)$, and since $T_{P_i(v)} (C_0)_v$ maps surjectively to $T_0
\bA^1$, we find that $D$ does not contain a component of $C_0$.

Let $f \in K[B]$ be such that $f(b_i)=0$ for $i=2,\ldots,n$ and $f(b_1)
\neq 0$, and define $h \in K[C_0]$ by $h(b,\tau,v):=f(b)+\tau$.  We set
$C:=C_0[1/h]$. Then $C$ is irreducible, affine, and flat over $S$; and
$\tau$ is a global uniformizer on $C$ for the point $P:=P_1$. Let $\rho \colon C
\to B$ be the natural map. It remains to show that the restriction of
$\rho$ to $C \setminus \im(P)$ is surjective. To this end, let $b \in
B$ and choose $\tau \neq 0$ such that $f(b)+\tau \neq 0$. Then solve
\eqref{eq:pisigmatau} for $v$ and note that $h(b,\tau,v)$ is nonzero,
so that $(b,\tau,v) \in C \setminus \im(P)$.
\end{proof}

\subsection{Proof of Theorem~\ref{thm:extend}} \label{ss:construct}

We now prove the theorem. Fix the following data:
\begin{itemize}
\item $X$ is an irreducible affine $\GL$-variety
\item $\phi \colon B \times \bA^{\umu} \to X$ is dominant, with $B$ an irreducible variety
\item $x$ is a $K$-point of $X$.
\end{itemize}
We fix an embedding $X \subset \bA^{\ulambda}$.

Let $\phi' \colon B' \times \bA^{\umu} \to X$ be as in
Proposition~\ref{prop:mapping}. Thus $B'$ is an irreducible smooth
variety, and we have an expression $x=\lim_{t \to 0} \phi'(\gamma(t),
v(t))$ where $\gamma$ is a $K\lbb t \rbb$-point of $B'$ and $v$ is a
bounded $K\lpp t \rpp$-point of $\bA^{\umu}$. For notational
simplicity, we now forget our original $(B, \phi)$ and drop the primes
from $(B', \phi')$. Let $n \ge 0$ be such that $t^n v(t)$ has no
negative powers of $t$, and let $m \gg 0$.

Let $(C/S, P, t, \rho)$ be the data produced by Proposition~\ref{prop:jets} applied to the $m$-jet defined by $\gamma$. Thus:
\begin{itemize}
\item $S$ is an irreducible affine variety;
\item $C \to S$ is a family of curves over $S$ (see \S \ref{ss:not}) such that the total
space $C$ is affine and irreducible;
\item $P \colon S \to C$ is a smooth point of $C$;
\item $t \in K[C]$ is a global uniformizer at $P$ that is a unit on $C
\setminus P$;
\item $\rho \colon C \to B$ is a map of varieties such that for any $K$-point $s \in S$ the map
\begin{displaymath}
\xymatrix@C=4em{
\Spec(K\lbb t \rbb/(t^{m+1})) \ar[r] & C_s \ar[r]^{\rho_s} \ar[r] & B }
\end{displaymath}
coincides with $\gamma$ to order $m$; and
\item the restriction of $\rho$ to $C \setminus P$ is surjective onto $B$.
\end{itemize}
Applying Proposition~\ref{prop:glob-lim}, we have a map
\begin{displaymath}
\alpha \colon C \times \fL^{\gamma}_{n,m} \to X, \qquad
\alpha(c, y) = \phi(\rho(c), y(t)).
\end{displaymath}
We extend this to a map
\begin{displaymath}
\psi_1 \colon C \times \fL^{\gamma}_{n,m} \times \bA^{\umu} \to X, \qquad
\psi_1(c, y(t), u) = \phi(\rho(c), y(t)+t^{m+1}u);
\end{displaymath}
again, for $c \in C \setminus P$ the right-hand side is evaluated by
substituting $t(c)$ for $t$, and for $c \in P$ an appropriate limit is
taken. Since $m$ is sufficiently large, adding $t^{m+1}u$ to $y(t)$ does not
affect the existence or the value of this limit. We are now ready to prove the theorem:

\begin{proof}[Proof of Theorem~\ref{thm:extend}]
Recall that $x=\lim_{t \to 0} \phi(\gamma(t), v(t))$. Thus $v$ defines
a point of $\fL^{\gamma}_{n,m}$. Let $\fL'$ be an irreducible
component of $\fL^{\gamma}_{n,m}$ containing $v$, let $Y=C \times \fL'
\times \bA^{\umu}$, and let $\psi \colon Y \to X$ be the restriction
of $\psi_1$ to $Y$. Since $C$, $\fL'$, and $\bA^{\umu}$ are each
irreducible $\GL$-varieties, so is $Y$. We have $x=\psi(P(s), v, 0)$ for any $s \in S(K)$, and so $x \in \im(\psi)$.

To complete the proof, we must show $\im(\phi) \subset \im(\psi)$.
Thus let $y$ be a $K$-point of $\im(\phi)$. Write $y=\phi(b, w)$ for
some $b \in B$ and $w \in \bA^{\umu}$. Let $c \in C \setminus P$
satisfy $\rho(c)=b$. Set $u=t(c)^{-(m+1)} \cdot (w-v(t(c)))$.
Note that $v(t)+t^{m+1} \cdot u$ evaluated at $c$ is equal to $w$. We thus see
\begin{displaymath}
\psi(c, v, u) = \phi(\rho(c), w) = b,
\end{displaymath}
which completes the proof.
\end{proof}

\section{Enlarging the image more} \label{s:enlarge2}

The purpose of \S \ref{s:enlarge2} is to prove the following theorem:

\begin{theorem} \label{thm:extend2}
Let $X$ be an irreducible $\GL$-variety, let $\phi \colon B
\times \bA^{\umu} \to X$ be a dominant morphism, and let $x$
be a scheme-theoretic point of $X$. Then there exists an
irreducible closed $\GL$-subvariety $Y$ of $\bA^r \times
\bA^{s \cdot \umu}$, for some $r$ and $s$, and a morphism $\psi \colon Y \to X$ of $\GL$-varieties such that $\im(\phi) \cup \{x\} \subset \im(\psi)$.
\end{theorem}

Theorem~\ref{thm:extend} and Theorem~\ref{thm:extend2} are nearly the same, but in the former $x$ is required to be a $K$-point, while in the latter it can be any scheme-theoretic point of $X$. This generality is crucially important in our application of the theorem. We deduce Theorem~\ref{thm:extend2} from Theorem~\ref{thm:extend} by working over larger fields. We start with a few lemmas.

\begin{lemma} \label{lem:rational-1}
Suppose we have the following:
\begin{itemize}
\item $\Omega$ is an algebraically closed field containing $K$
\item $Y'$ is a closed $\GL$-subvariety of $\bA^{\ulambda}_{\Omega}$ (over $\Omega$).
\end{itemize}
Then we can find:
\begin{itemize}
\item an irreducible affine variety $S$ (over $K$) equipped with a generic point $\eta \colon \Spec(\Omega) \to S$
\item a closed $\GL$-subvariety $Y$ of $S \times \bA^{\ulambda}$ (over $K$)
\end{itemize}
such that $Y_{\eta}=Y'$.
\end{lemma}

\begin{proof}
Let $R$ be the coordinate ring of $\bA^{\ulambda}$, so that $\Omega
\otimes_K R$ is the coordinate ring of $\bA^{\ulambda}_{\Omega}$. Let
$I' \subset \Omega \otimes_K R$ be the vanishing ideal of $Y'$. By
\cite{draisma}, there are finitely many elements $f_1, \ldots, f_r$ of
$\Omega \otimes_K R$ such that $I'$ is the radical of the $\GL$-ideal
generated by the $f_i$'s. Let $A$ be a finitely generated $K$-subalgebra
of $\Omega$ such that each $f_i$ belongs to $A \otimes_K R$, and let $I
\subset A \otimes_K R$ be the radical of the $\GL$-ideal generated by the
$f_i$'s. Note that $I'$ is the extension of $I$; the key point is that
the extension of $I$ is radical. (It is clear that the
extension of $I$ to $\Frac(A) \otimes_K R$ is radical, and for
the further extension from $\Frac(A) \otimes_K R$ to $\Omega
\otimes_K R$ one uses that the characteristic is zero.) Let
$S=\Spec(A)$ and let $Y \subset S \times \bA^{\ulambda}$ be the vanishing
locus of $I$. The inclusion $A \subset \Omega$ corresponds to a generic
point $\eta \colon \Spec(\Omega) \to S$. We have $Y_{\eta}=Y'$ since $I'$
is the extension of $I$.
\end{proof}

\begin{lemma} \label{lem:rational-2}
Suppose we have the following:
\begin{itemize}
\item $\Omega$ is an algebraically closed field containing $K$
\item $S$ is an irreducible affine variety over $K$ equipped with a generic point $\eta \colon \Spec(\Omega) \to S$
\item $X$ and $Y$ are affine $\GL$-varieties over $S$.
\item $\alpha \colon Y_{\eta} \to X_{\eta}$ is a map of $\GL$-varieties over $\Omega$.
\end{itemize}
Then we can find:
\begin{itemize}
\item an irreducible affine variety $S'$ equipped with a map $S' \to S$ and a generic point $\eta' \colon \Spec(\Omega) \to S'$ lifting $\eta$
\item a map $\beta \colon Y_{S'} \to X_{S'}$ of $\GL$-schemes over $S'$
\end{itemize}
such that the fiber of $\beta$ over $\eta'$ is $\alpha$.
\end{lemma}

\begin{proof}
Let $A=K[S]$ be the coordinate ring of $S$. The point $\eta$ corresponds to an injection of $K$-algebras $A \to \Omega$. The coordinate rings $K[X]$ and $K[Y]$ of $X$ and $Y$ are naturally $A$-algebras. The given map $\alpha$ corresponds to a homomorphism
\begin{displaymath}
\alpha^* \colon \Omega \otimes_A K[X] \to \Omega \otimes_A K[Y]
\end{displaymath}
of $\Omega$-algebras. Let $A'$ be a finitely generated $A$-subalgebra of $\Omega$ such that $\alpha^*(K[X])$ is contained in $A' \otimes_A K[Y]$. This exists since $K[X]$ is finitely $\GL$-generated. Then $\alpha^*$ restricts to a homomorphism of $A'$-algebras
\begin{displaymath}
\beta^* \colon A' \otimes_A K[X] \to A' \otimes_A K[Y].
\end{displaymath}
We take $S'=\Spec(A')$ and let $\beta \colon Y_{S'} \to X_{S'}$ be the map induced by $\beta^*$. Of course, $\eta'$ corresponds to the inclusion $A' \to \Omega$. (Note that $Y_{S'}$ and $X_{S'}$ may not be reduced, which is why we say $\GL$-scheme instead of $\GL$-variety.)
\end{proof}

\begin{lemma} \label{lem:rational-3}
Suppose we have the following:
\begin{itemize}
\item $\Omega$ is an algebraically closed field containing $K$
\item $X$ is an irreducible affine $\GL$-variety (over $K$)
\item $Y'$ is an irreducible closed $\GL$-subvariety of $\bA^{\ulambda}_{\Omega}$ (over $\Omega$)
\item $\alpha \colon Y' \to X_{\Omega}$ is a map of $\GL$-varieties (over $\Omega$).
\end{itemize}
Then we can find
\begin{itemize}
\item an irreducible affine variety $S$ over $K$ equipped with a generic point $\eta \colon \Spec(\Omega) \to S$
\item an irreducible closed $\GL$-subvariety $Y$ of $S \times \bA^{\ulambda}$ (over $K$) with $Y_{\eta}=Y'$
\item a map $\gamma \colon Y \to X$ of $\GL$-varieties (over $K$)
\end{itemize}
such that $\alpha$ is the fiber over $\eta$ of the map $\gamma \times \pi \colon Y \to X \times S$, where $\pi \colon Y \to S$ is the projection.
\end{lemma}

\begin{proof}
By Lemma~\ref{lem:rational-1}, we can find:
\begin{itemize}
\item an irreducible affine variety $S_0$ over $K$ with a generic point $\eta_0 \colon \Spec(\Omega) \to S_0$
\item a closed $\GL$-subvariety $Y_0$ of $S_0 \times \bA^{\ulambda}$
\end{itemize}
such that $(Y_0)_{\eta}=Y'$. Let $X_0=X \times S_0$. Then $Y_0$ and $X_0$ are $\GL$-varieties over $S_0$, and $\alpha$ defines a map of their fibers over $\eta_0$. By Lemma~\ref{lem:rational-2}, we can find:
\begin{itemize}
\item an irreducible affine variety $S$ over $K$ equipped with a map $S \to S_0$ and a generic point $\eta \colon \Spec(\Omega) \to S$ lifting $\eta_0$
\item a map $\beta \colon (Y_0)_S \to (X_0)_S$ of $\GL$-schemes over $S$
\end{itemize}
such that $\alpha$ is the fiber of $\beta$ over $\eta$. Let $Y_1$ be the reduced subscheme of $(Y_0)_S$; this is a $\GL$-variety, and $(Y_1)_{\eta}=Y'$ since $Y'$ is reduced. Note that $(X_0)_S = X \times S$, and so $\beta \vert_{Y_1}$ has the form $\gamma_1 \times \pi_1$, where $\gamma_1 \colon Y_1 \to X$ is a map and $\pi_1 \colon Y_1 \to S$ is the projection. Let $Y_1^1, \ldots, Y_1^r$ be the irreducible components of $Y_1$. Since $Y'=(Y_1^1)_{\eta} \cup \cdots \cup (Y_1^r)_{\eta}$ and $Y'$ is irreducible, we must have $Y'=(Y_1^i)_{\eta}$ for some $i$. We can thus take $Y=Y_1^i$, and $\gamma$ to be the restriction of $\gamma_1$ to $Y$.
\end{proof}

\begin{proof}[Proof of Theorem~\ref{thm:extend2}]
Let $\phi \colon B \times \bA^{\umu} \to X$ and $x \in X$ be given as in Theorem~\ref{thm:extend2}. Let $\Omega$ be an algebraically closed field containing $K$ such that $x$ is defined over $\Omega$. Consider the base change of $\phi$ to $\Omega$:
\begin{displaymath}
\phi_{\Omega} \colon B_{\Omega} \times \bA^{\umu}_{\Omega} \to X_{\Omega}.
\end{displaymath}
Let $x'$ be an $\Omega$-point of $X_{\Omega}$ that maps to $x$. By Theorem~\ref{thm:extend} (applied with $\Omega$ as the base field), we can find an irreducible closed $\GL$-subvariety $Y'$ of $\bA^r_{\Omega} \times \bA^{s \cdot \umu}_{\Omega}$ and a morphism $\psi' \colon Y' \to X_{\Omega}$ such that $\im(\phi_{\Omega}) \cup \{x'\} \subset \im(\psi')$.

Applying Lemma~\ref{lem:rational-3}, we can find:
\begin{itemize}
\item an irreducible affine variety $S$ over $K$ equipped with a generic point $\eta \colon \Spec(\Omega) \to S$.
\item an irreducible closed $\GL$-subvariety $Y$ of $S
\times \bA^r \times \bA^{s \cdot \umu}$ (over $K$) with $Y_{\eta}=Y'$
\item a map $\psi \colon Y \to X$ of $\GL$-varieties (over $K$)
\end{itemize}
such that $\psi'$ is the fiber over $\eta$ of $\psi \times \pi$. Note that we can embed $S$ into $\bA^m$ for some $m$, which realizes $Y$ as a closed subvariety of $\bA^{r+m} \times \bA^{s \cdot \umu}$.

Consider the following diagram
\begin{displaymath}
\xymatrix@C=4em{
Y_{\Omega} \ar[r] \ar[d]_{\psi'} & Y \ar[d]^{\psi \times \pi} \\
X_{\Omega} \ar[r]^{\id \times \eta} & X \times S \ar[r] & X }
\end{displaymath}
Here the middle column is a map of $\GL$-varieties over $S$, and the left column is its fiber over the point $\eta$; the bottom right map is the projection onto the first factor. Now, $x' \in X_{\Omega}$ belongs to the image of $\psi'$ and maps to $x$ in $X$. Chasing the diagram, we see that $x \in \im(\psi)$. Now suppose that $z \in \im(\phi)$. Write $z=\phi(y)$ with $y \in B \times \bA^{\umu}$, and let $y'$ be a pre-image of $y$ in $B_{\Omega} \times \bA^{\umu}_{\Omega}$. Then $z'=\phi_{\Omega}(y')$ is a pre-image of $z$ and belongs to $\im(\phi_{\Omega})$, and therefore to $\im(\psi')$. Chasing the diagram again, we see that $z \in \im(\psi)$. This completes the proof.
\end{proof}

\section{Pulling off the top piece} \label{s:top}

\subsection{The main result} \label{ss:top}

For a tuple $\ulambda$, let $\Vert \ulambda \Vert$ be the maximum size of a partition in $\ulambda$. Let $X$ be a $\GL$-variety. We define the \defn{embedding degree} of $X$, denoted $\ed(X)$, to be the minimal $d$ for which there is a closed immersion $X \to \bA^{\ulambda}$ for some tuple $\ulambda$ with $\Vert \ulambda \Vert=d$. We define the \defn{unirational degree} of $X$, denoted $\ud(X)$, to be the minimal $d$ for which there is a dominant morphism $B \times \bA^{\ulambda} \to X$ for some tuple $\ulambda$ with $\Vert \ulambda \Vert=d$ and variety $B$. The unirationality theorem ensures that $\ud(X) \le \ed(X)$.

The following theorem is the main result of \S \ref{s:top}:

\begin{theorem} \label{thm:small}
Let $X$ be an irreducible affine $\GL$-variety with $\ed(X)=d$. Then we
can find a surjective map of $\GL$-varieties $Y \times \bA^{\ul{\tau}} \to
X$ where $Y$ is an irreducible affine $\GL$-variety with $\ud(Y)<d$ and
$\ed(Y) \le d$ and every partition appearing in $\ul{\tau}$ has size $d$.
\end{theorem}

Let us explain the significance of this theorem. The main result of this paper is the strong unirationality theorem (Theorem~\ref{thm:stronguni}). The core argument in the proof of that theorem applies when we have a strict inequality $\ud<\ed$; see \S \ref{ss:sketch}(b). If we start with an arbitrary $\GL$-variety $X$, we will appeal to the above theorem and then apply the core argument to $Y$.

\subsection{A structural result} \label{ss:struct}

Fix, for the duration of \S \ref{s:top}, a tuple $\ulambda$ with $\Vert \ulambda \Vert=d$. Write
\begin{displaymath}
\bA^{\ulambda} = \bA^{\ul{\alpha}} \times \prod_{i=1}^r (V_i \otimes \bA^{\beta_i})
\end{displaymath}
where $\Vert \ul{\alpha} \Vert<d$, each $\beta_i$ has size $d$, the
$\beta_i$ are distinct, and $V_i$ is a finite dimensional vector space. Note that $V_i \otimes \bA^{\beta_i}$ is simply a product of $\dim{V_i}$ copies of $\bA^{\beta_i}$. We now show that a $\GL$-subvariety of $\bA^{\ulambda}$ is, in a sense, built out of two pieces: one with $\ud<d$, and one that is linear. These two pieces will eventually yield the $Y$ and $\bA^{\ul{\tau}}$ in Theorem~\ref{thm:small}. 

\begin{proposition} \label{prop:struct}
Let $X$ be an irreducible closed $\GL$-subvariety of $\bA^{\ulambda}$. There exists:
\begin{itemize}
\item an irreducible smooth projective variety $C$,
\item a vector subbundle $\cR_i$ of $C \times V_i$ for each $1 \le i \le r$, and
\item an irreducible closed $\GL$-subvariety $W$ of $C \times \bA^{\ulambda}$ with $\ud(W)<d$
\end{itemize}
such that, letting $W^+ \subset C \times \bA^{\ulambda}$ be the Zariski closure of $W + \prod_{i=1}^r (\cR_i \otimes \bA^{\beta_i})$, the projection map $C \times \bA^{\ulambda} \to \bA^{\ulambda}$ restricts to a surjection $W^+ \to X$.
\end{proposition}

We think of $\cR_i \otimes \bA^{\beta_i}$ as a group over $C$: the fiber over $c \in C$ is a group, and it acts on $\{c\} \times \bA^{\ulambda}$.
The proof of Proposition~\ref{prop:struct} will take all of \S \ref{ss:struct}. Apply the unirationality theorem to obtain a dominant morphism
\begin{displaymath}
\phi \colon B \times \bA^{\umu} \times \prod_{i=1}^r (U_i \otimes \bA^{\beta_i}) \to X \subset \bA^{\ulambda}
\end{displaymath}
where $B$ is an irreducible affine variety, $\umu$ is a pure tuple with $\Vert \umu \Vert<d$, and the $U_i$'s are finite dimensional vector spaces. The following lemma decomposes $\phi$ into two pieces. This roughly corresponds to the decomposition of $X$ in the proposition.

\begin{lemma} \label{lem:struct-1}
There exists a morphism $\psi \colon B \times \bA^{\umu} \to \bA^{\ulambda}$ of $\GL$-varieties, and morphisms $\lambda^i \colon B \to \Hom_K(U_i, V_i)$ of varieties, such that
\begin{displaymath}
\phi(b,x, y_1, \ldots, y_r) = \psi(b,x)+\sum_{i=1}^r (\lambda^i(b) \otimes \id_{\bA^{\beta_i}})(y_i)
\end{displaymath}
for $b \in B$, $x \in \bA^{\umu}$, and $y_i \in U_i \otimes \bA^{\beta_i}$.
\end{lemma}

\begin{proof}
Consider the $\bA^{\ul{\alpha}}$ component of $\phi$. Since $\Vert
\ul{\alpha} \Vert<d= \vert \beta_i \vert$, the only $\GL$-equivariant map $\bA^{\beta_i} \to \bA^{\ul{\alpha}}$ is the zero map. Thus the $\bA^{\ul{\alpha}}$ component of $\phi$ must factor through $B \times \bA^{\umu}$.

Now consider the $V_i \otimes \bA^{\beta_i}$ component of $\phi$. This
must be a sum of some function on $B \times \bA^{\umu}$ and a function
on $B \times (U_i \otimes \bA^{\beta_i})$ that is linear in the second
component. This is exactly what the lemma asserts.
\end{proof}

A standard argument shows that there is a non-empty open
affine subset of $B$ on which the rank of each $\lambda^i$
is constant (and maximal), say equal to $d_i$. Replace $B$ with this open 
affine. Let $C=\Gr_{d_1}(V_1) \times \cdots \times \Gr_{d_r}(V_r)$, where $\Gr$ denotes the Grassmanian of subspaces, and let $\cR_i$ be the tautological subbundle of $C \times V_i$. We have a morphism
\begin{displaymath}
\rho \colon B \to C, \qquad \rho(b) = (\im{\lambda^1(b)}, \ldots,
\im{\lambda^r(b)}).
\end{displaymath}
Consider the map
\begin{align*}
\theta \colon B \times \bA^{\umu} \times \prod_{i=1}^r (U_i \otimes \bA^{\beta_i}) &\to C \times \bA^{\ulambda} \\
(b, x, y_1, \ldots, y_r) &\mapsto (\rho(b), \phi(b,x, y_1, \ldots, y_r))
\end{align*}
We let $W_0 \subseteq C \times \bA^{\ulambda}$ be the image of $B \times \bA^{\umu} \times \{0\}$ under $\theta$, and let $W$ be the closure of $W_0$. Note that $W_0$ is irreducible, and so $W$ is as well, and furthermore $\ud(W)<d$. We have now defined $C$, $\cR_i$, and $W$ as in the statement of Proposition~\ref{prop:struct}. Let $W_0^+=W_0+\prod_{i=1}^r(\cR_i \otimes \bA^{\beta_i})$, and let $W^+$ be as in the proposition.

\begin{lemma} \label{lem:struct-2}
We have $\im(\theta)=W_0^+$.
\end{lemma}

\begin{proof}
By Lemma~\ref{lem:struct-1}, we have
\begin{displaymath}
\theta(b, x, y_1, \ldots, y_r)
= (\rho(b), \psi(b,x) + \sum_{i=1}^r (\lambda^i(b) \otimes
\id_{\bA^{\beta_i}})(y_i)).
\end{displaymath}
If we set all $y_i$ equal to zero, this expression lies in $W_0$, 
while the $i$th term in the sum belongs to the fiber of 
$\cR_i \otimes \bA^{\beta_i}$ over $\rho(b)$; thus the entire expression belongs to $W_0^+$. This proves $\im(\theta) \subset W_0^+$.

Now let $(c,u)$ with $c \in C$ and $u \in \bA^{\ulambda}$ be a given element
of $W_0^+$. Then we can write $u=v+\sum_{i=1}^r
r_i$, where $(c,v) \in W_0$ and each $r_i$ lies in the fiber of
$\cR_i \otimes \bA^{\beta_i}$ over $c$.
Write $(c,v)=\theta(b,x,0)$ for some $b \in B$ and $x \in
\bA^{\umu}$; note that $c=\rho(b)$. By the definition of $\cR_i$, we
can find $y_i \in V_i \otimes \bA^{\beta_i}$ such that $r_i=(\rho(b),
(\lambda^i(b) \otimes \id_{\bA^{\beta_i}})(y_i))$. We thus see that $p=\theta(b, x, y_1, \ldots, y_r)$, which completes the proof.
\end{proof}

The following lemma completes the proof of the proposition:

\begin{lemma}
The projection $p_2 \colon C \times \bA^{\ulambda} \to \bA^{\ulambda}$ maps $W^+$ surjectively to $X$.
\end{lemma}

\begin{proof}
We have $\phi = p_2 \circ \theta$. By Lemma~\ref{lem:struct-2}, we see that $\im(\phi)=p_2(W_0^+)$. Since $\im(\phi)$ is a dense subset of $X$, and $W_0^+$ is a dense subset of $W^+$, we see that $p_2(W^+)$ is a dense subset of $X$. Since $W^+$ is closed and the map $p_2$ is proper (since $C$ is projective), it follows that $p_2(W^+)=X$, as required.
\end{proof}

\subsection{An improved structural result}

We now make some slight improvements on Proposition~\ref{prop:struct}:

\begin{proposition} \label{prop:struct2}
Let $X$ be an irreducible closed $\GL$-subvariety of $\bA^{\ulambda}$.
There exist:
\begin{itemize}
\item an irreducible smooth affine variety $C$,
\item for each $1 \leq i \leq r$ a direct summand $\cR_i$ of the vector
bundle $C \times V_i$ on $C$ that is also trivial as a vector bundle, and
\item an irreducible closed $\GL$-subvariety $W$ of $C \times \bA^{\ulambda}$ with $\ud(W)<d$
\end{itemize}
such that, letting $W^+ \subset C \times \bA^{\ulambda}$ be the
Zariski closure of $W + \prod_{i=1}^r (\cR_i \otimes \bA^{\beta_i})$,
the projection map $C \times \bA^{\ulambda} \to \bA^{\ulambda}$
restricts to a surjection $W^+ \to X$.
\end{proposition}

The differences with Proposition~\ref{prop:struct} are that $C$ is now
affine and each $\cR_i$ is a trivial bundle. To begin, let $C$, $\cR_i$, $W$, and $W^+$ be as in
Proposition~\ref{prop:struct}. We will first show that the properties
we care about are preserved by base change along certain morphisms $C'
\to C$, and we will then construct a base change that proves
Proposition~\ref{prop:struct2}.

\begin{lemma} \label{lem:struct2-1}
Let $f \colon C' \to C$ be a surjective smooth morphism with (geometrically) irreducible fibers, and let $\cR'_i$ and $W'$ be the pull-backs of $\cR_i$ and $W$ under $f$. Let $(W')^+$ be the Zariski closure of $W' + \prod_{i=1}^r (\cR_i' \times \bA^{\beta_i})$. Then:
\begin{enumerate}
\item $W'$ is irreducible and $\ud(W')<d$.
\item The projection map $(W')^+ \to X$ is surjective.
\end{enumerate}
\end{lemma}

\begin{proof}
(a) We have a dominant morphism $\psi \colon B \times \bA^{\umu} \to W$ with
$\Vert \umu \Vert<d$. Let $B'=B \times_C C'$, so that $\psi$ pulls
back to $\psi' \colon B' \times \bA^{\umu} \to W'$. Since $B' \to B$ is a smooth morphism
with irreducible fibers and $B$ is irreducible, it follows that $B'$ is irreducible\footnote{This is
well-known, but we recall the proof. Let $U$ and $V$ be non-empty open subsets of $B'$. Let $\ol{U}$ and $\ol{V}$ be their images in $B$, which are open open since the map is smooth. Since $B$ is irreducible,
it follows that $\ol{U}$ and $\ol{V}$ intersect. Let $x$ be a point of intersection. Then $U \cap B_x$ and
$V \cap B_x$ are non-empty open subsets of the irreducible space $B_x$, and thus intersect. Thus $U$
and $V$ intersect. Since all non-empty opens in $B$ intersect, it follows that $B$ is irreducible.}.
Similarly for $W'$. Since $\psi$ is dominant, its image contains a non-empty open subset of $W$. It follows that the image
of $\psi'$ contains a non-empty open subset of $W'$. Since $W'$ is irreducible, we thus see that $\psi'$
is dominant. Replacing $B'$ with an affine open if necessary, we find $\ud(W')<d$.

(b) Since any base change of $f$ is an open map, it follows that formation of the image-closure of a map
commutes with base change along $f$. Hence $(W')^+$ is the base change of $W^+$, and so $(W')^+$ surjects onto $W^+$, and thus onto $X$.
\end{proof}

\begin{lemma} \label{lem:struct2-2}
We can choose $f \colon C' \to C$ such that the following conditions hold:
\begin{enumerate}
\item $C'$ is an irreducible affine variety,
\item $f$ is a smooth surjective map with (geometrically) irreducible fibers, and
\item $f^*(\cR_i)$ is a trivial bundle and a direct summand of $C' \times V_i$, for all $i$.
\end{enumerate}
\end{lemma}

\begin{proof}
Write $\cV_i$ for the trivial vector bundle $C \times V_i$ over $C$,
and write $d_i$ and $n_i=\dim(V_i)$ for the ranks of $\cR_i$ and $\cV_i$,
respectively. We write $K^{d_i}$ for the subspace $K^{d_i} \times
\{0\}^{n_i-d_i}$ of $K^{n_i}$. 
Let $\uHom(\cR_i \subseteq \cV_i,K^{d_i} \subseteq K^{n_i})$ be the vector bundle on $C$ whose
fiber over a point $c$ is the space of linear maps $V_i \to
K^{n_i}$ that map $(\cR_i)_c$ into $K^{d_i}$; 
and let $\uIsom(\cR_i \subseteq \cV_i,K^{d_i} \subseteq K^{n_i})$ be the open subset of
$\uHom(\cR_i \subseteq \cV_i,K^{d_i} \subseteq K^{n_i})$ where that
linear map is a linear isomorphism. The structure map $\uIsom(\cR_i
\subseteq \cV_i,K^{d_i} \subseteq K^{n_i}) \to C$ is surjective,
smooth, and has irreducible geometric fibers (the fibers are
isomorphic to the (parabolic) stabilizer in $\GL_{n_i}$ of the 
subspace $K^{d_i}$). Put
\begin{displaymath}
C_1=\prod_{i=1}^r \uIsom(\cR_i \subseteq \cV_i, K^{d_i} \subseteq K^{n_i})
\end{displaymath}
where the product is over $C$. Then $f_1\colon C_1 \to C$ is surjective,
smooth, and has irreducible geometric fibers. Furthermore, the pullback
$f_1^*(\cR_i)$ of $\cR_i$ to $C_1$ is trivialized by the map $f_1^*(\cR_i)
\to C_1 \times K^{d_i}$ that sends a point $(c,\phi_1,\ldots,\phi_r,v)$
to $(c,\phi_1,\ldots,\phi_r,\phi_i(v)) \in C_1 \times K^{d_i}$.
Furthermore, let $\cS_i$ be the bundle over $C_1$ whose fibre over
$(c,\phi_1,\ldots,\phi_r)$ equals $\phi_i^{-1}(\{0\}^{d_i} \times
K^{n_i-d_i}) \subseteq V_i$. Then $f_1^*(\cV_i)$ is isomorphic to the
direct sum of the bundles $f_1^*(\cR_i)$ and $\cS_i$. Now, apply
Jouanolou's trick: we can find an affine variety $C'$ and a morphism $C'
\to C_1$ that is surjective, smooth, and has irreducible geometric
fibers. Naturally, the pull-back of $f_1^*(\cR_i)$ to $C'$ remains
trivial, and the decomposition $f_1^*(\cR_i) \oplus \cS_i$ 
pulls back to a direct sum decomposition of the trivial bundle
$C' \times V_i$, as desired. 
\end{proof}

\begin{proof}[Proof of Proposition~\ref{prop:struct2}]
Let $f$ be as in Lemma~\ref{lem:struct2-2}. This yields the desired set-up by Lemma~\ref{lem:struct2-1}.
\end{proof}

\subsection{Proof of Theorem~\ref{thm:small}}

Let $X$ be a closed $\GL$-subvariety of
$\bA^{\ulambda}=\bA^{\ul{\alpha}} \times \prod_{i=1}^r (V_i \otimes
\bA^{\beta_i})$, where $\Vert \alpha \Vert < d$, each $\beta_i$ has
size $d$, the $\beta_i$ are distinct, and the $V_i$ are
finite-dimensional vector spaces. 
Apply
Proposition~\ref{prop:struct2} to get an irreducible smooth affine variety
$C$, for every $1 \leq i \leq r$ a direct sum decomposition $\cV_i=\cR_i
\oplus \cS_i$ of the trivial bundle $\cV_i=C \times V_i$ where $\cR_i$
is itself trivial as a vector bundle, and an irreducible closed $\GL$-subvariety $W$
of $C \times \bA^{\ulambda}$ with $\ud(W)<d$ such that the Zariski
closure $W^+$ of $W + \prod_{i=1}^r (\cR_i \otimes \bA^{\beta_i})$
maps surjectively to $X$. Set 
\begin{displaymath}
Y:=W^+ \cap \left( \bA^{\ul{\alpha}} \times \prod_{i=1}^r \cS_i \otimes
\bA^{\beta_i} \right) \subseteq C \times \bA^{\ulambda},
\end{displaymath}
where the $r$-fold product is over $C$. By definition, $Y$ is closed.
Furthermore, the projections $\cV_i \to \cS_i$ with kernel $\cR_i$ induce a
dominant map $W \to Y$ of $\GL$-varieties, so that $Y$ is irreducible with
$\ud(Y) \leq \ud(W)<d$. Let $\psi_i \colon C
\times K^{d_i} \to \cR_i$ be a trivialization. Then the map of
$\GL$-varieties   
\begin{align*}
Y \times \prod_{i=1}^r K^{d_i} \otimes \bA^{\beta_i} &\to W^+ \\
(y,q_1,\ldots,q_r) &\mapsto y+(\psi_1 \otimes
\id_{\bA^{\beta_1}}(q_1),\ldots,\psi_r \otimes
\id_{\bA^{\beta_r}}(q_r))
\end{align*}
is surjective. Composing this with the surjection $W^+ \to X$ yields
the theorem.

\section{Proofs of the main theorems} \label{s:uni}

We are now ready to prove
Theorems~\ref{thm:stronguni}--\ref{thm:GeneralizedOrbit}.

\subsection{Improved unirationality}

The following theorem implies Theorem~\ref{thm:stronguni}, and provides a bit more information about the map.

\begin{theorem} \label{thm:uni-improved}
Let $X$ be an irreducible affine $\GL$-variety. Then there is a surjective map of $\GL$-varieties $B \times \bA^{\umu} \to X$ for some irreducible variety $B$ and pure tuple $\umu$. Moreover, one can take $\Vert \umu \Vert=\ud(X)$.
\end{theorem}

\begin{proof}
We proceed by induction on the embedding degree. Thus suppose that $d$ is given and the theorem holds for $X$ with $\ed(X)<d$. We now prove the theorem for $X$ with $\ed(X) \le d$.

\textit{\textbf{Step 1:} the case $\ud(X)<d$.} Let $X$ be an irreducible
affine $\GL$-variety with $\ed(X) \le d$ and $\ud(X)<d$. By the
unirationality theorem, there is
a dominant morphism $\phi \colon B \times \bA^{\umu} \to X$ for some
irreducible affine variety $B$ and some pure tuple $\umu$ with $\Vert
\umu \Vert = \ud(X)$. Of all such $\phi$, choose one for which the
interior of $\im(\phi)$ is maximal. Such a maximal $\phi$ exists since
the ascending chain condition holds for $\GL$-stable open subsets of $X$
by noetherianity. We show that $\phi$ is surjective, which will establish
the theorem for $X$.

Suppose, by way of contradiction, that $\phi$ is not surjective. Let $x$ be the generic point of an irreducible component of $X \setminus \int(\im(\phi))$. By Theorem~\ref{thm:extend}, there exists an irreducible closed $\GL$-subvariety $Y$ of $\bA^r \times \bA^{s \cdot \umu}$, for some $r$ and $s$, and a morphism $\psi \colon Y \to X$ such that $\im(\phi) \cup \{x\} \subset \im(\psi)$. We have
\begin{displaymath}
\ed(Y) \le \Vert s \cdot \umu \Vert \le \Vert \umu \Vert = \ud(X)<d,
\end{displaymath}
and so the inductive hypothesis implies that the theorem holds for $Y$. We can thus choose a surjective morphism $\theta \colon C \times \bA^{\unu} \to Y$ for some irreducible affine variety $C$ and some tuple $\unu$ with $\Vert \unu \Vert=\ud(Y) \le \ed(Y) \le \ud(X)$. We have $\im(\phi) \cup \{x\} \subset \im(\psi \circ \theta)$. By Lemma~\ref{lem:interior} below, it follows that $\int(\im(\phi)) \cup \{x\} \subset \int(\im(\psi \circ \theta))$, contradicting the maximality of $\int(\im(\phi))$. We conclude that $\phi$ is surjective.

\textit{\textbf{Step 2:} the case $\ud(X)=d$.} Now let $X$ be an
irreducible affine $\GL$-variety with $\ed(X)=\ud(X)=d$. According to
Theorem~\ref{thm:small}, there exists a surjection $Y \times
\bA^{\ul{\tau}} \to X$ where $Y$ is an irreducible affine
$\GL$-variety with $\ed(Y) \le d$ and $\ud(Y) <d$, and $\ul{\tau}$ is
some tuple of partitions of size $d$. By Step~1, we have a surjection
$B \times \bA^{\umu} \to Y$ for some irreducible affine variety $B$
and tuple $\umu$ with $\Vert \umu \Vert < d$. We thus obtain a surjection $B \times \bA^{\umu \cup \ul{\tau}} \to X$, as required.
\end{proof}

Recall \cite[\S 7.4]{polygeom} that a subset of a $\GL$-variety is called \defn{$\GL$-constructible} if it is a finite union of locally closed $\GL$-stable subsets.

\begin{lemma} \label{lem:interior}
Let $X$ be an irreducible $\GL$-variety, let $A \subset X$ be a $\GL$-constructible subset, let $x$ be the generic point of an irreducible component of $X \setminus \int(A)$, and let $B \subset X$ be a $\GL$-constructible subset with $A \cup \{x\} \subset B$. Then $x \in \int(B)$.
\end{lemma}

\begin{proof}
The set $X \setminus B$ is a $\GL$-constructible subset of the closed set $X \setminus \int(A)$ which does not contain the point $x$, which is a generic point of an irreducible component of $X \setminus \int(A)$. It follows that $x$ does not belong to $\ol{X \setminus B} = X \setminus \int(B)$, and so $x \in \int(B)$.
\end{proof}

We note one corollary of the theorem.

\begin{corollary}
Let $X$ be an irreducible $\GL$-variety and let $Y$ be an irreducible closed $\GL$-subvariety. Then $\ud(Y) \le \ud(X)$.
\end{corollary}

\begin{proof}
Let $d=\ud(X)$. Applying Theorem~\ref{thm:uni-improved}, let $\phi \colon B \times \bA^{\umu} \to X$ be a surjective map of $\GL$-varieties, where $\Vert \umu \Vert=d$. Let $Z=\phi^{-1}(Y)$, which is a closed $\GL$-subvariety of $B \times \bA^{\umu}$. Since $\phi$ is surjective, it follows that the induced map $Z \to Y$ is surjective, and so there is some irreducible component $Z_1$ of $Z$ that maps dominantly to $Y$. We have $\ud(Z_1) \le \ed(Z_1) \le d$ and, since $Z_1 \to Y$ is dominant, $\ud(Y) \le \ud(Z_1)$. The result thus follows.
\end{proof}

\subsection{Existence of curves}

In \cite[Theorem~1.3]{imgclosure}, we saw that in an irreducible
$\GL$-variety one can join a point to an open set by an irreducible
curve. We now prove the much stronger result Theorem~\ref{thm:Curves}
about joining any two points by a curve:

\begin{proof}[Proof of Theorem~\ref{thm:Curves}]
Let $x,y$ be $K$-points of the irreducible $\GL$-variety $X$.
By Theorem~\ref{thm:uni-improved} there exists a surjection $\phi \colon
B \times \bA^{\umu} \to X$ of $\GL$-varieties with $B$ an irreducible
variety. Let $\tilde{x}$ and $\tilde{y}$ be inverse images of $x$ and
$y$. Let $i \colon C \to B \times \bA^{\umu}$ be a map from an irreducible
curve with $\tilde{x}, \tilde{y} \in \im(i)$; it is easy to find such
a map, see \cite[Lemma 5.2]{imgclosure}. Now $j=\phi \circ i$ is a
curve passing through $x$ and $y$.
\end{proof}

\subsection{Irreducibility of mapping spaces}

We now prove Theorem~\ref{thm:MappingSpace}.

\begin{proof}[Proof of Theorem~\ref{thm:MappingSpace}]
Fix a pure tuple $\ulambda$. We show that the mapping space
$\cM_{\ulambda}(X)$ is irreducible. 
Invoking Theorem~\ref{thm:uni-improved}, choose a surjective morphism $\phi \colon B \times \bA^{\umu} \to X$ with $B$ irreducible and $\umu$ pure. We claim that the induced map
\begin{displaymath}
\phi_* \colon \cM_{\ulambda}(B \times \bA^{\umu}) \to \cM_{\ulambda}(X)
\end{displaymath}
is surjective. Since $\cM_{\ulambda}(B \times \bA^{\umu})=B
\times \cM_{\ulambda}(\bA^{\umu})$ is a vector bundle over
$B$ by~\cite[Proposition 2.3(c)]{polygeom}, and thus irreducible, this will show that $\cM_{\ulambda}(X)$ is irreducible.

Let $x \in \cM_{\ulambda}(X)$ be a (scheme-theoretic) point, let $C
\subset \cM_{\ulambda}(X)$ be the corresponding irreducible
closed subvariety, and let $\psi \colon C \times
\bA^{\ulambda} \to X$ be the canonical map. By~\cite[Proposition 6.4]{polygeom}, we can find a commutative diagram
\begin{displaymath}
\xymatrix@C=4em{
D \times \bA^{\ulambda} \ar[r]^-{\alpha \times \id} \ar[d] &
C \times \bA^{\ulambda} \ar[d]^{\psi} \\
B \times \bA^{\umu} \ar[r]^{\phi} & X }
\end{displaymath}
where $\alpha$ is dominant. This yields a commutative diagram
\begin{displaymath}
\xymatrix@C=4em{
D \ar[r]^-{\alpha} \ar[d] & C \ar[d] \\
\cM_{\ulambda}(B \times \bA^{\umu}) \ar[r]^{\phi_*} & \cM_{\ulambda}(X) }
\end{displaymath}
Via the top path, the generic point of $D$ maps to $x \in \cM_{\ulambda}(X)$. Thus, via the other path, we see that $x \in \im(\phi_*)$.
\end{proof}

\begin{remark}
In Theorem~\ref{thm:uni-improved}, one can now take $B=\cM_{\umu}(X)$.
\end{remark}

\begin{remark} \label{rmk:uni-not-ac}
Suppose $K$ is not algebraically closed, and let $X$ be a
geometrically irreducible affine $\GL$-variety. Since formation of
mapping spaces is compatible with base change, we see that
$\cM_{\ulambda}(X)$ is geometrically irreducible for any $\ulambda$.
For appropriate $\umu$, the map $\cM_{\umu}(X) \times \bA^{\umu} \to
X$ is surjective over $\ol{K}$, and thus surjective. We thus see that
Theorem~\ref{thm:uni-improved} holds over $K$, provided that we assume
that $X$ is geometrically irreducible. 
\end{remark}

\subsection{Type space and the space of generalized orbits} \label{ss:type}

Let $X$ be an irreducible $\GL$-variety. Recall from \S \ref{ss:main} that $X^{\orb}$ is the space of generalized orbits on $X$. The space $X^\orb$ is filtered by subspaces $X^{\orb}_{\umu}$, indexed by tuples $\umu$, as follows. The generalized orbit of $x \in X$ is in $X^\orb_{\umu}$ if and only if there exists a finite dimensional irreducible affine variety $B$ and a dominant $\GL$-equivariant morphism $B \times \bA^{\umu} \to \ol{\GL \cdot x}$. (See \cite[\S 8]{polygeom} for the description of this set in terms of types, as described in the introduction.) The set $X^\orb_{\umu}$ inherits the subspace topology from $X^\orb$. By the unirationality theorem, $X^\orb=\bigcup_{\umu} X^\orb_{\umu}$. 

Let $\Gamma_{\umu} \subseteq \cM_{\umu}(\bA^{\umu})$ denote the automorphism group of $\bA^{\umu}$; by \cite{polygeom} this is a finite-dimensional, connected, affine algebraic group, and it acts naturally on the finite dimensional variety $\cM_{\umu}(X)$. Let $X^{\type}_{\umu}$ be the set of scheme-theoretic points of $\cM_{\umu}(X)$ fixed by $\Gamma_{\umu}$. The space $X^{\type}_{\umu}$ inherits the subspace topology from $\cM_{\umu}(X)$. There is a natural surjective map $\cM_{\ulambda}(X) \to X^{\type}_{\ulambda}$ that maps a point to the generic point of the closure of its $\Gamma_{\ulambda}$-orbit, and it turns out that the (subspace) topology of $X^{\type}_{\ulambda}$ agrees with the quotient topology coming from this map; see \cite[\S8.7]{polygeom}.

In \cite[\S8.8]{polygeom} we introduce the map $\rho_{\ulambda} \colon X^{\type}_{\ulambda} \to X^\orb_{\ulambda}$ as follows: given $x \in X^{\type}_{\ulambda}$, let $B$ be its closure in $\cM_{\ulambda}(X)$, and let $\rho_{\ulambda}(x)$ be the image of the generic point under the canonical morphism $B \times \bA^{\ulambda} \to X$. By \cite[Theorem 8.25]{polygeom}, $\rho_{\ulambda}$ is a continuous bijection. We can now prove more, namely, the following more explicit version of Theorem~\ref{thm:GeneralizedOrbit}.

\begin{theorem}
Let $X$ be an affine $\GL$-variety and let $\ulambda$ be a pure tuple. Then the continuous bijection $\rho_{\ulambda} \colon X^{\type}_{\ulambda} \to X^{\orb}_{\ulambda}$ is a homeomorphism.
\end{theorem}

\begin{proof}
Let $x,y \in X^{\orb}_{\ulambda}$ be points such that $y$ is a
specialization of $x$. Let $x',y' \in X^{\type}_{\ulambda}$ be the inverse images of $x$ and $y$ under $\rho$. We must show that $y'$ is a specialization of $x'$.

Let $Z \subset X$ be the Zariski closure of $x$. Consider the diagram
\begin{displaymath}
\xymatrix{
Z^{\type}_{\ulambda} \ar[r]^{\rho_{\ulambda}} \ar[d] & Z^{\orb}_{\ulambda} \ar[d] \\
X^{\type}_{\ulambda} \ar[r]^{\rho_{\ulambda}} & X^{\orb}_{\ulambda} }
\end{displaymath}
The vertical maps are induced by the inclusion $Z \to X$, and easily seen to be continuous; moreover, the diagram is commutative. We regard $x'$ and $y'$ as elements of $Z_{\ulambda}^{\type}$; note that right vertical map is an injection, and so the left vertical map is too since the horizontal maps are bijections.

Now, $\cM_{\ulambda}(Z)$ is irreducible
(Theorem~\ref{thm:MappingSpace}), and so
$Z_{\ulambda}^{\type}$, being a quotient of
$\cM_{\ulambda}(Z)$, is irreducible as well. Thus $\rho_{\ulambda}
\colon Z_{\ulambda}^{\type} \to Z_{\ulambda}^{\orb}$ is a continuous
bijection of irreducible sober spaces. Since $x'$ maps to the generic
point $x$ of $Z_{\ulambda}^{\orb}$, it follows from the
Lemma~\ref{lem:generic} below that $x'$ is the generic point of
$Z_{\ulambda}^{\type}$. Thus $y'$ is a specialization of $x'$. Since
the map $Z_{\ulambda}^{\type} \to X_{\ulambda}^{\type} \to X^{\type}$
is continuous, this relation holds in $X^{\type}_{\ulambda}$ as well.
\end{proof}

\begin{lemma} \label{lem:generic}
Let $f \colon A \to B$ be a continuous bijection of irreducible sober spaces. Then $f$ maps the generic point of $A$ to the generic point of $B$.
\end{lemma}

\begin{proof}
Let $a \in A$ be the generic point and let $Z \subset B$ be the closure of $\{f(a)\}$. Then $f^{-1}(Z)$ is a closed subset of $A$ containing $a$, and is therefore all of $A$. Since $f$ is a bijection, it follows that $Z=B$. Thus $f(a)$ is dense in $B$, and therefore the generic point.
\end{proof}

\begin{example}
Let $X=\bA^{[(2)]}$ be the space of quadratic forms. The generalized orbits are in bijection with $\{0,1,2,\ldots,\infty\}$ via rank. We thus identify $X^{\orb}$ with the set $\{0,1,\ldots,\infty\}$. It is not difficult to see that the closed sets are exactly those of the form $\{0,1,\ldots,n\}$ for some natural number $n$, together with $X^{\orb}$ itself. In particular, we note that the set of generalized orbits of finite rank is dense in $X^{\orb}$.

Let $\umu=[(1),(1)]$. The mapping space $\cM_{\umu}(X)$ is three-dimensional: for any field extension $\Omega \supseteq K$, an $\Omega$-valued point is of the form $(f,g) \mapsto af^2+bfg+cg^2$ where $a,b,c \in \Omega$. Here $\Gamma_{\umu}=\GL_2$, acting by replacing $(f,g)$ by two linear combinations. The only $\Gamma_{\umu}$-fixed points in $\cM_{\umu}(X)$ are: the generic point, the generic point of the subscheme where the discriminant $b^2-4ac$ is zero, and the point with $a=b=c=0$. The space $X^{\orb}_{\umu}$ is $\{0,1,2\}$. It is not difficult to see directly that $\rho_{\umu}$ is a homeomorphism in this case.

Here, the theorem may not be so surprising. This lack of surprise is due to the fact that in finite dimensions, $\GL_n$ has only finitely many orbits on quadrics in $n$ variables. But it is truly remarkable that the space of generalized orbits in \emph{any} $\GL$-variety admits such a description using finite-dimensional pieces!
\end{example}


\begin{thebibliography}{BBOV2}

\bibitem[AH]{ah}
Tigran Ananyan, Melvin Hochster. Small subalgebras of polynomial rings and Stillman's conjecture. 
\textit{J.~Amer.~Math.~Soc.} \textbf{33} (2020), no.~1, pp.~2910--309. \DOI{10.1090/jams/932} \arxiv{1610.09268v1}

\bibitem[BBOV1]{bbov}
Edoardo Ballico, Arthur Bik, Alessandro Oneto, Emanuele Ventura. Strength and slice rank of
forms are generically equal. \textit{Israel J.~Math.} (2022). 
\DOI{10.1007/s11856-022-2397-0} \arxiv{2102.11549}

\bibitem[BBOV2]{bbov2}
Edoardo Ballico, Arthur Bik, Alessandro Oneto, Emanuele Ventura. The set of forms with
bounded strength is not closed. \textit{Compt.~Rend.~Math.} \textbf{360} (2022), pp.~371--380. 
\DOI{10.5802/crmath.302} \arxiv{2012.01237}

\bibitem[BDES1]{polygeom}
Arthur Bik, Jan Draisma, Rob H.~Eggermont, Andrew Snowden. The
geometry of polynomial representations. \textit{Int.\ Math.\ Res.\
Not.\ IMRN} \textbf{16} (2023), pp.~14131--14195. 
\DOI{10.1093/imrn/rnac220} \arxiv{2105.12621}

\bibitem[BDES2]{imgclosure}
Arthur Bik, Jan Draisma, Rob H.~Eggermont, Andrew Snowden. Uniformity
for limits of tensors. \arxiv{2305.19866}

\bibitem[BDS]{poschar}
Arthir Bik, Jan Draisma, Andrew Snowden. The geometry of polynomial
representations in positive characteristic. \textit{Math.~Z.}
\textbf{310}(1) (2025), paper no.~13, 39 pp. \DOI{10.1007/s00209-025-03720-y} \arxiv{2406.07415}

\bibitem[DES]{des} Harm Derksen, Rob H.~Eggermont, Andrew Snowden. Topological noetherianity for cubic polynomials. \textit{Algebra Number Theory} \textbf{11} (2017), no.~9, pp.~2197--2212. \DOI{10.2140/ant.2017.11.2197} \arxiv{1701.01849}

\bibitem[Dr]{draisma}
Jan Draisma. Topological Noetherianity of polynomial functors. \textit{J.\ Amer.\ Math.\ Soc.} \textbf{32} (2019), no.~3, pp.~691--707. \DOI{10.1090/jams/923} \arxiv{1705.01419}

\bibitem[KaZ]{kazi}
David Kazhdan, Tamar Ziegler. 
Applications of algebraic combinatorics to algebraic geometry. 
\textit{Indag.\ Math., New Ser.} \textbf{32} (2021), no.~6, pp.~1412--1428. \DOI{10.1016/j.indag.2021.09.002} \\\arxiv{2005.12542}

\bibitem[La]{landsberg}
Joseph M.~Landsberg. \textit{Tensors: geometry and applications.} Graduate Studies in Mathematics 128, American Mathematical Society, Providence, RI, 2012.

\bibitem[NSS]{sym2noeth}
Rohit Nagpal, Steven V Sam, Andrew Snowden. Noetherianity of some degree two twisted commutative algebras. \textit{Selecta Math.\ (N.S.)} \textbf{22} (2016), no.~2, pp.~913--937. \DOI{10.1007/s00029-015-0205-y} \arxiv{1501.06925}

\bibitem[S]{schmidt}
W.~M.~Schmidt. The density of integer points on homogeneous varieties. \textit{Acta Math.}
\textbf{154} (1985), no.~3--4, pp.~243--296. \DOI{10.1007/BF02392473}

\bibitem[Stacks]{stacks}
Stacks Project. {\tiny\url{http://stacks.math.columbia.edu}} (accessed May, 2026).

\bibitem[Wa]{wadsworth}
Adrian R.~Wadsworth. Hilbert subalgebras of finitely generated algebras. \textit{J.\ Algebra} \textbf{43} (1976), pp.~298--304. \DOI{10.1016/0021-8693(76)90161-7}

\end{thebibliography}
\end{document}